\documentclass[10pt]{amsart}

\usepackage[dvipsnames]{xcolor}
\usepackage{amsmath,amsthm,amssymb}
\usepackage{enumitem}
\usepackage{tikz}
\usepackage{hyperref}
\usepackage{amsfonts,marvosym,amscd}
\usepackage{mathtools}
\usepackage{mathscinet}
\usepackage{wasysym} 
\usepackage{graphicx}
\usepackage{psfrag}
\usepackage{datetime}
\usepackage{tikz-cd}  
\usepackage{geometry} 
\usepackage[all]{xy}
\usepackage{url}
\usepackage{pifont}

\usepackage{tabstackengine}
\setstackgap{L}{.75\baselineskip}

\usetikzlibrary{decorations.pathmorphing}
\usetikzlibrary{positioning}
\usetikzlibrary{calc}

\numberwithin{equation}{section}

\newtheoremstyle{nwstyle}
  {15pt} % Space above
  {15pt} % Space below
  {\itshape} % Body font
  {} % Indent amount
  {\bfseries} % Theorem head font
  {.} % Punctuation after theorem head
  {.5em} % Space after theorem head
  {} % Theorem head spec (can be left empty, meaning `normal')

\theoremstyle{nwstyle}
\newtheorem{theorem}{Theorem}[section]
\newtheorem{lemma}[theorem]{Lemma}
\newtheorem{proposition}[theorem]{Proposition}
\newtheorem{corollary}[theorem]{Corollary}

\newtheorem{innercustomthm}{Theorem}
\newenvironment{customthm}[1]
  {\renewcommand\theinnercustomthm{#1}\innercustomthm}
  {\endinnercustomthm}

\theoremstyle{definition}
\newtheorem{definition}[theorem]{Definition}

\newtheorem{example}[theorem]{Example}

\theoremstyle{remark}
\newtheorem{remark}[theorem]{Remark}

\newcommand{\nonconsec}[2]{\prescript{\circlearrowleft}{}{\mathbf{I}_{#1}^{#2}}}

\newcommand{\modset}[2]{{\mathbf{I}_{#1}^{#2}}}
\newcommand{\chainset}[2]{{\mathbf{J}_{#1}^{#2}}}
\newcommand{\subs}[3]{\binom{[#1]}{#2}_{#3}}

\newcommand{\imunder}{\dashrightarrow}
\newcommand{\simp}{\mathring{e}}

\newcommand{\tcs}[1]{\tabbedCenterstack{#1}}
\newcommand{\floor}[1]{\lfloor #1 \rfloor}
\newcommand{\ceil}[1]{\lceil #1 \rceil}
\newcommand{\bigfloor}[1]{\Bigl\lfloor #1\Bigr\rfloor}
\newcommand{\bigceil}[1]{\Bigl\lceil #1\Bigr\rceil}

\newcommand{\cmark}{\ding{51}}
\newcommand{\xmark}{\ding{55}}

\newcommand{\modules}{\mathrm{mod}\,}

\let\emptyset\varnothing

\usepackage[backend=biber,
	style=alphabetic,
	url=false,
	doi=false,
	isbn=false,
	maxbibnames=99]{biblatex}
\DeclareFieldFormat{title}{#1} % Removes quotation marks around title
\renewbibmacro{in:}{} % Removes ``In:'' before journal etc
\title{New interpretations of the higher Stasheff--Tamari orders}
\author{Nicholas J. Williams}
\address{School of Mathematics and Actuarial Science, University of Leicester, University Road, LE1 7RH}
\email{njw40@le.ac.uk}
\subjclass[2010]{Primary: 05E10; Secondary: 06A07, 52B11}
\keywords{Cyclic polytopes, higher Auslander--Reiten theory, higher Stasheff--Tamari orders, triangulations}

\begin{document}

\begin{abstract}
In 1996, Edelman and Reiner defined the two higher Stasheff--Tamari orders on triangulations of cyclic polytopes and conjectured them to coincide. We open up an algebraic angle for approaching this conjecture by showing how these orders arise naturally in the representation theory of the higher Auslander algebras of type $A$, denoted $A_{n}^{d}$. For this we give new combinatorial interpretations of the orders, making them comparable. We then translate these combinatorial interpretations into the algebraic framework. We also show how triangulations of odd-dimensional cyclic polytopes arise in the representation theory of $A_{n}^{d}$, namely as equivalence classes of maximal green sequences. We furthermore give the odd-dimensional counterpart to the known description of $2d$-dimensional triangulations as sets of non-intersecting $d$-simplices of a maximal size. This consists in a definition of two new properties which imply that a set of $d$-simplices produces a $(2d+1)$-dimensional triangulation.
\end{abstract}

\maketitle

\tableofcontents

\section{Introduction}

The primary aims of this paper are three-fold. Firstly, we give new combinatorial interpretations of the two higher Stasheff--Tamari orders which clarify their natures and make them comparable to each other. Secondly, we give a novel description of triangulations of odd-dimensional cyclic polytopes which gives the counterpart to the even-dimensional description in \cite{ot}. Finally, we show how the higher Stasheff--Tamari orders and triangulations of odd-dimensional cyclic polytopes arise in the algebraic framework of the higher Auslander algebras of type $A$ from \cite{ot}. This opens up an algebraic angle for approaching the conjecture of Edelman and Reiner that the two higher Stasheff--Tamari orders coincide.

\emph{Cyclic polytopes} are the analogues of convex polygons in higher dimensions \cite{gale}. For this reason, cyclic polytopes possess important properties \cite{mcmullen} and are found in many areas of mathematics, such as higher category theory and algebraic $K$-theory \cite{kv-poly,dk-segal,poguntke,djw}, as well as game theory \cite{svs06,svs16}.

There is much that is unknown about cyclic polytopes. In general, there is no formula for the number of triangulations of a cyclic polytope, although specific cases have been solved \cite{as-triang-num}. Such a formula would generalise the Catalan numbers, which count numerous combinatorial objects \cite{stanley-catalan}, including triangulations of convex polygons, binary trees, and bracketings. These Catalan objects are ordered by the \emph{Tamari lattice} \cite{tamari,huang-tamari}, a ubiquitous mathematical structure which arises in many fields, such as computer science, topology, and physics \cite{tamari-festschrift}.

In 1991, Kapranov and Voevodsky defined a higher-dimensional version of the Tamari lattice on triangulations of cyclic polytopes \cite[Definition 3.3]{kv-poly}. Edelman and Reiner then developed this by defining two orders on triangulations of cyclic polytopes, known as the \emph{higher Stasheff--Tamari orders} \cite{er}, where the first higher Stasheff--Tamari order is the same as the order defined by Kapranov and Voevodsky \cite[Proposition 3.3]{thomas-bst}. In their paper, Edelman and Reiner conjectured the two higher Stasheff--Tamari orders to coincide but, despite much work \cite{rambau,err,thomas-bst,thomas,rr}, this conjecture remains open today.

In this paper, we open up a new algebraic approach for proving this conjecture by showing how the higher Stasheff--Tamari orders arise within the representation theory of algebras. Oppermann and Thomas show that triangulations of even-dimensional cyclic polytopes are related to the representation theory of $A_{n}^{d}$ \cite{ot}, the higher Auslander algebras of type $A$. These algebras were introduced by Iyama within the programme of higher Auslander--Reiten theory \cite{iy-aus,iy-high-ar,iy-clus}. This is a new and active area of research within representation theory which has found connections with non-commutative algebraic geometry \cite{himo} and homological mirror symmetry \cite{djl}. The relation between triangulations of even-dimensional cyclic polytopes and the representation theory of $A_{n}^{d}$ is a higher-dimensional version of the relation between triangulations of convex polygons and cluster categories of type $A$. Cluster categories were introduced in \cite{bmrrt} as a categorification of the cluster algebras of Fomin and Zelevinsky \cite{fz1}, which have found applications to Poisson geometry \cite{gsv03,gsv05}, dynamical systems \cite{fz-laurent}, and Teichm\"uller theory \cite{fg06,fg09}. Cluster categories are a powerful tool and have been used to solve open problems in mathematical physics \cite{kel-per}.

Our work in this paper reveals that the connection between triangulations of cyclic polytopes and higher Auslander--Reiten theory is richer than previously known. Not only do triangulations of odd-dimensional cyclic polytopes play a role, but the higher Stasheff--Tamari orders arise naturally in this algebraic framework. This provides new insights into the combinatorial structure of higher Auslander--Reiten theory and paves the way for an algebraic proof of the equivalence of the orders.

In order to produce these algebraic interpretations of the orders, we require new combinatorial interpretations. One reason that the Edelman--Reiner conjecture is difficult is that the definitions of the higher Stasheff--Tamari orders make them hard to compare. The first order is more combinatorial in nature; it is defined by means of covering relations, which are given by \emph{increasing bistellar flips}. A bistellar flip in two dimensions consists of flipping a diagonal in a quadrilateral. The second order is more geometric in nature. Triangulations of a cyclic polytope induce characteristic sections of the cyclic polytope one dimension larger. The second order holds when the characteristic section of one triangulation is above the characteristic section of the other. The advantage of the first order is that its covering relations are clear; the advantage of the second order is that one can directly compute when it holds, rather than having to search for a sequence of increasing bistellar flips. The equivalence of the two orders would unite both of these properties.

The second order ought to admit a combinatorial definition, since it is independent of the geometric realisation of the polytope. A beautiful such definition was given in \cite{thomas}, but this uses a different framework to triangulations, and does not conveniently accommodate the first order. Other combinatorial descriptions of triangulations of cyclic polytopes have conveniently accommodated the first order but not the second \cite{rambau}. In this paper we enable comparison between the orders in a way previously impossible by giving combinatorial descriptions of both orders in the same framework. We use these results to verify the Edelman--Reiner conjecture in several new cases by writing programs in Sage. Furthermore, we use these programs to find a counter-example to the conjecture that the first order is a lattice. As a corollary to the new combinatorial interpretations, we give non-recursive realisations of the minimal embeddings of the second order into Boolean lattices \cite{thomas}.

The key insight in this paper which makes these descriptions possible is that the higher Stasheff--Tamari orders behave differently in odd and even dimensions, which ought therefore to be treated separately. Our first result gives the following new combinatorial descriptions for the first and second higher Stasheff--Tamari orders in even dimensions. We use $\leqslant_{1}$ to denote the first order and $\leqslant_{2}$ to denote the second. We always describe the first order using its covering relations, which are denoted $\lessdot_{1}$. This makes the statements of the results simpler. We refer the reader to Section \ref{back-cyc} for the definition of intertwining $d$-simplices.

\begin{customthm}{A}[Theorem \ref{main-thm-1st} and Theorem \ref{main-thm-2nd}]\label{thm-int-even-comb}
For triangulations $\mathcal{T}$ and $\mathcal{T}'$ of a $2d$-dimensional cyclic polytope, we have that
\begin{enumerate}
\item $\mathcal{T} \lessdot_{1} \mathcal{T}'$ if and only if $\mathcal{T}'$ is the result of replacing a $d$-simplex of $\mathcal{T}$ by one which it intertwines; and
\item $\mathcal{T} \leqslant_{2} \mathcal{T}'$ if and only if no $d$-simplex of $\mathcal{T}'$ intertwines a $d$-simplex of $\mathcal{T}$.
\end{enumerate}
\end{customthm}

The interpretation of the higher Stasheff--Tamari orders for odd dimensions is as follows.

\begin{customthm}{B}[Theorem \ref{thm-hst1-odd} and Corollary \ref{thm-hst2-odd}]\label{thm-int-odd-comb}
For triangulations $\mathcal{T}$ and $\mathcal{T}'$ of a $(2d+1)$-dimensional cyclic polytope, we have that
\begin{enumerate}
\item $\mathcal{T} \lessdot_{1} \mathcal{T}'$ if and only if the set of $d$-simplices of $\mathcal{T}'$ is the result of removing a $d$-simplex from the set of $d$-simplices of $\mathcal{T}$; and 
\item $\mathcal{T} \leqslant_{2} \mathcal{T}'$ if and only if the set of $d$-simplices of $\mathcal{T}$ contains the set of $d$-simplices of $\mathcal{T}'$.
\end{enumerate}
\end{customthm}

Triangulations of even-dimensional cyclic polytopes are better understood than those in odd dimensions. A triangulation of a $2d$-dimensional cyclic polytope can be described as a set of non-intersecting $d$-simplices of maximal size \cite{ot}, just as a triangulation of a convex polygon is a maximal set of non-intersecting arcs. To describe odd-dimensional triangulations we define two new properties for collections of $d$-simplices which we call being \emph{supporting} (Definition \ref{def-support}) and being \emph{bridging} (Definition \ref{def-bridge}). We show that if a collection of $d$-simplices is supporting and bridging, then one can build a triangulation of a $(2d+1)$-dimensional cyclic polytope out of them. This gives the following theorem.  

\begin{customthm}{C}[Theorem \ref{thm-class-odd-dim}]\label{thm-int-odd-desc}
Triangulations of the cyclic polytope $C(m,2d+1)$ are given by sets of $d$-simplices which are supporting and bridging.
\end{customthm}

The new combinatorial descriptions from Theorems \ref{thm-int-even-comb} and \ref{thm-int-odd-comb} allow us to give our algebraic interpretations of the orders. Triangulations of the $2d$-dimensional cyclic polytope with $n+2d$ vertices correspond to tilting $A_{n}^{d}$-modules \cite{ot}. (Tilting modules are modules which induce a weak form of equivalence between algebras \cite{happel}.) For the definition of $A_{n}^{d}$, the higher Auslander algebra of type $A$, see Section \ref{subsect-high-ar}. In even dimensions, the higher Stasheff--Tamari orders induce classical orders on tilting modules introduced in \cite{rs-simp}, as stated in the following theorem.  This was already known for the special case of the Tamari lattice \cite{buan-krause,thomas_tamari}. For the definition of a left mutation and of $\prescript{\bot}{}T$ see Section \ref{subsect-high-ar}.

\begin{customthm}{D}[Theorem \ref{alg-cor-1st} and Theorem \ref{alg-cor-2nd}]\label{thm-int-even-alg}
Let $\mathcal{T}$ and $\mathcal{T}'$ be triangulations of a $2d$-dimensional cyclic polytope corresponding to tilting modules $T$ and $T'$ over $A_{n}^{d}$, the higher Auslander algebra of type $A$. We then have that
\begin{enumerate}
\item $\mathcal{T} \lessdot_{1} \mathcal{T}'$ if and only if $T'$ is a left mutation of $T$; and
\item $\mathcal{T} \leqslant_{2} \mathcal{T}'$ if and only if $\prescript{\bot}{}T \subseteq \prescript{\bot}{}T'$.
\end{enumerate}
\end{customthm}

There is an alternative algebraic framework from \cite{ot}, where triangulations of the $2d$-dimensional cyclic polytope with $n+2d+1$ vertices correspond to so-called cluster-tilting objects for $A_{n}^{d}$. Cluster-tilting objects first arose in cluster categories, where they are the analogues of tilting modules. There also exists a version of Theorem \ref{thm-int-even-alg} in the cluster-tilting framework, which is Theorem \ref{thm-clus-tilt}. 

In odd dimensions we show that triangulations of cyclic polytopes correspond to equivalence classes of maximal green sequences (Theorem \ref{thm-max-green-desc}). Maximal green sequences were introduced in the context of Donaldson--Thomas invariants in mathematical physics \cite{kel-green}. We define higher-dimensional \emph{$d$-maximal green sequences} as sequences of mutations of cluster-tilting objects from the projectives to the shifted projectives, see Section \ref{sect-max-green}. The orders induced in odd dimensions by the higher Stasheff--Tamari orders are very natural, but have not previously been considered. The Edelman--Reiner conjecture here corresponds to a stronger form of the ``no-gap'' conjecture made in \cite{bdp}, cases of which were proven in \cite{g-mc,hi-no-gap}.

The algebraic description of the higher Stasheff--Tamari orders in odd dimensions is as follows. For the definition of the equivalence relation see Section \ref{sect-odd}, and for the definition of an increasing elementary polygonal deformation see Section \ref{odd-first}.

\begin{customthm}{E}[Theorem \ref{thm-max-green}]\label{thm-int-odd-alg}
Let $\mathcal{T}$ and $\mathcal{T}'$ be triangulations of a $(2d+1)$-dimensional cyclic polytope corresponding to equivalence classes of $d$-maximal green sequences $[G]$ and $[G']$ of $A_{n}^{d}$, the higher Auslander algebra of type $A$. We then have that
\begin{enumerate}
\item $\mathcal{T} \lessdot_{1} \mathcal{T}'$ if and only if $[G']$ is an increasing elementary polygonal deformation of $[G]$; and
\item $\mathcal{T} \leqslant_{2} \mathcal{T}'$ if and only if the set of summands of $[G]$ contains the set of summands of $[G']$.
\end{enumerate}
\end{customthm}

Theorem \ref{thm-int-odd-alg} therefore belongs to the cluster-tilting framework, because $d$-maximal green sequences are defined in terms of cluster-tilting objects. As before, there is a version of Theorem \ref{thm-int-odd-alg} in the tilting framework, which is Theorem \ref{thm-alg-odd-hst1}. In this introduction we state Theorem \ref{thm-int-even-alg} in the tilting framework and Theorem \ref{thm-int-odd-alg} in the cluster-tilting framework, because these are the most natural frameworks for the respective theorems.

A corollary of Theorem \ref{thm-int-odd-alg} is that the set of equivalence classes of maximal green sequences of linearly oriented $A_{n}$ is a lattice (Corollary \ref{cor-a-green-lat}). This is because in dimension 3 the two orders are known to be equivalent and known to be lattices \cite{er}.

This paper is structured as follows. In Section \ref{background}, we give relevant background on cyclic polytopes, their triangulations, and the higher Stasheff--Tamari orders. We also describe briefly the higher Auslander--Reiten theory of Iyama. We divide the paper according to odd and even dimensions. In Section \ref{sect-hst} we give our combinatorial interpretations of the higher Stasheff--Tamari orders for even-dimensional cyclic polytopes. We then show how these translate to orders on tilting modules and cluster-tilting objects. We split this section in half, dealing with the first order in Section \ref{even-first} and the second order in Section \ref{even-second}. In Section \ref{sect-odd-desc}, we classify triangulations of $(2d+1)$-dimensional cyclic polytopes in terms of their sets of $d$-simplices. In Section \ref{sect-odd} we give combinatorial descriptions of the higher Stasheff--Tamari orders for odd-dimensional cyclic polytopes, similarly taking each order in turn. We then show how these orders translate to orders on maximal chains of tilting modules and maximal green sequences. We report on the results of our computer programs in Section \ref{sect-comp}.

\subsection*{Acknowledgements}

This paper forms part of my PhD studies. I would like to thank my supervisor Professor Sibylle Schroll for comments on earlier drafts of this paper, as well as for her continuing support and attention. I am grateful to the University of Leicester for funding my PhD in support of her grant from the EPSRC (reference EP/P016294/1). I thank Hugh Thomas for interesting conversations and for making the final observation of Corollary \ref{cor-a-green-lat}. I would also like to thank Vic Reiner and Jordan McMahon for helpful suggestions, and Hipolito Treffinger for bringing the paper \cite{hi-no-gap} to my attention.

\section{Background}\label{background}

In this section we give necessary information on cyclic polytopes, higher Auslander--Reiten theory, and the relation between them shown in \cite{ot}. First, we set out some notational conventions.

We use $[m]$ to denote the set $\{1, \dots, m\}$. By $\binom{[m]}{k}$ we mean the set of subsets of $[m]$ of size $k$. For convenience, if $A=\{a_{0}, \dots, a_{d}\} \subseteq [m]$, where $a_{0}<\dots <a_{d}$, we write $A=(a_{0}, \dots, a_{d})$. We also refer to $A$ as a \emph{$(d+1)$-tuple} here. We always follow the convention that if $A\subseteq [m]$ and $\# A=k+1$, then the elements of $A$ are labelled $A=(a_{0}, \dots, a_{k})$. The same applies to other letters of the alphabet: the upper case letter denotes the tuple; the lower case letter is used for the elements, which are ordered according to their index starting from 0.

\subsection{Cyclic polytopes}\label{back-cyc}

Cyclic polytopes should be thought of as higher-dimensional analogues of convex polygons. General introductions to this class of polytopes can be found in \cite[Lecture 0]{ziegler} and \cite[4.7]{gruenbaum}. According to Gr\"unbaum, the construction of them in current use is due to Gale \cite{gale} and Klee \cite{klee}. They were introduced and studied in the 1950s by Gale \cite{gale-abstract} and Motzkin \cite{motzkin}. (The earlier work of Carath\'eodory \cite{car1907,car1911} is related, but the convex bodies studied in these papers are not cyclic polytopes.)

A subset $X \subset \mathbb{R}^{n}$ is \emph{convex} if for any $x,x' \in X$, the line segment connecting $x$ and $x'$ is contained in $X$. The \emph{convex hull} $\mathrm{conv}(X)$ of $X$ is the smallest convex set containing $X$ or, equivalently, the intersection of all convex sets containing $X$.

The \emph{moment curve} is defined by $p(t):=(t, t^{2}, \dots , t^{\delta}) \subset \mathbb{R}^{\delta}$ for $t \in \mathbb{R}$, where $\delta \in \mathbb{N}_{\geqslant 1}$. Choose $t_{1}, \dots , t_{m} \in \mathbb{R}$ such that $t_{1} < t_{2} < \dots < t_{m}$. The convex hull $\mathrm{conv}(p(t_{1}), \dots , p(t_{m}))$ is a \emph{cyclic polytope} $C(m, \delta)$. More generally, if we have a subset $(x_{0}, \dots, x_{r})=X \subseteq [m]$, then we write $C(X,\delta)$ for the convex hull $\mathrm{conv}(p(t_{x_{0}}), \dots, p(t_{x_{r}}))$.

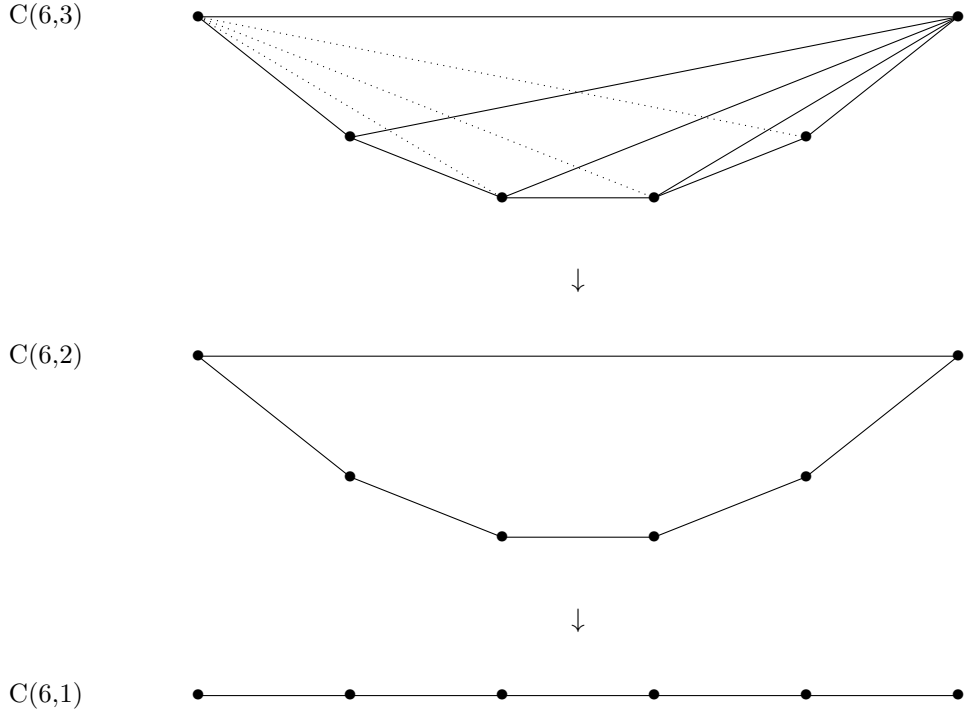
\begin{figure}
\caption{The cyclic polytopes $C(6,1), C(6,2), C(6,3)$}\cite[Figure 2]{er}
\[
\begin{tikzpicture}
% Coordinates for C(6,1)
\coordinate(11) at (-5,0);
\coordinate(12) at (-3,0);
\coordinate(13) at (-1,0);
\coordinate(14) at (1,0);
\coordinate(15) at (3,0);
\coordinate(16) at (5,0);

% Coordinates for C(6,2)
\coordinate(21) at (-5,4.5);
\coordinate(22) at (-3,2.9);
\coordinate(23) at (-1,2.1);
\coordinate(24) at (1,2.1);
\coordinate(25) at (3,2.9);
\coordinate(26) at (5,4.5);

% Coordinates for C(6,3)
\coordinate(31) at (-5,9);
\coordinate(32) at (-3,7.4);
\coordinate(33) at (-1,6.6);
\coordinate(34) at (1,6.6);
\coordinate(35) at (3,7.4);
\coordinate(36) at (5,9);

% Projection arrows
\node(21-ar) at (0,1){$\downarrow$};
\node(32-ar) at (0,5.5){$\downarrow$};

% Polytope labels
\node(1) at (-7,0){C(6,1)};
\node(2) at (-7,4.5){C(6,2)};
\node(3) at (-7,9){C(6,3)};

% Edges for C(6,1)
\draw (11) -- (12) -- (13) -- (14) -- (15) -- (16);

% Edges for C(6,2)
\draw (21) -- (22) -- (23) -- (24) -- (25) -- (26) -- (21);

% Edges for C(6,3)
\draw (31) -- (32) -- (33) -- (34) -- (35) -- (36) -- (31);
\draw (36) -- (34);
\draw (36) -- (33);
\draw (36) -- (32);
\draw[dotted] (31) -- (33);
\draw[dotted] (31) -- (34);
\draw[dotted] (31) -- (35);

% Bullets at vertices
\node at (11){$\bullet$};
\node at (12){$\bullet$};
\node at (13){$\bullet$};
\node at (14){$\bullet$};
\node at (15){$\bullet$};
\node at (16){$\bullet$};
\node at (21){$\bullet$};
\node at (22){$\bullet$};
\node at (23){$\bullet$};
\node at (24){$\bullet$};
\node at (25){$\bullet$};
\node at (26){$\bullet$};
\node at (31){$\bullet$};
\node at (32){$\bullet$};
\node at (33){$\bullet$};
\node at (34){$\bullet$};
\node at (35){$\bullet$};
\node at (36){$\bullet$};
\end{tikzpicture}
\]
\end{figure}

A \emph{triangulation} of a cyclic polytope $C(m, \delta)$ is a subdivision of $C(m, \delta)$ into $\delta$-dimensional simplices whose vertices are elements of $\{p_{t_{1}}, \dots, p_{t_{m}}\}$. We identify a triangulation of $C(m,\delta)$ with the corresponding set of $\delta$-simplices. We write $\mathcal{S}(m,\delta)$ for the set of triangulations of $C(m,\delta)$.

As explained in \cite{rambau}, whether or not a collection of $(\delta+1)$-subsets of $\{t_{1}, \dots, t_{m}\}$ gives the collection of $\delta$-simplices of a triangulation is independent of the values of $t_{1} < \dots < t_{m}$ defining $C(m, \delta)$. We therefore fix $t_{i}=i \in \mathbb{R}$ throughout the paper and use $[m]$ as the vertex set of $C(m,\delta)$. One can thus describe $k$-dimensional simplices in $C(m, \delta)$ as $(k + 1)$-tuples with elements in $[m]$, where $k \leqslant \delta$.  A triangulation can be specified by giving the collection of $(\delta + 1)$-tuples corresponding to the $\delta$-simplices of the triangulation. We refer to the $k$-dimensional faces of a simplex as \emph{$k$-faces}. Given a $k$-tuple $A \subseteq [m]$, we use $|A|_{\delta}$ to refer to the geometric simplex $\mathrm{conv}(A)$ in dimension $\delta$. When the dimension is clear, we will drop the subscript. We will also drop the inner pair of brackets when we take the geometric realisation of an explicit tuple so that $|a_{0}, \dots, a_{d}| := |(a_{0}, \dots, a_{d})|$. Additionally, we always use Roman letters to refer to tuples and Greek letters to refer to geometric simplices. (We relax the distinction between tuples and geometric simplices in the introduction in order to make simpler statements of our main theorems.)

A \emph{circuit} of a cyclic polytope $C(m,\delta)$ is a pair, $(Z_{+}, Z_{-})$, of disjoint sets of vertices of $C(m,\delta)$ which are inclusion-minimal with the property that $\mathrm{conv}(Z_{+}) \cap \mathrm{conv}(Z_{-}) \neq \emptyset$. It follows that $\mathrm{conv}(Z_{+})$ and $\mathrm{conv}(Z_{-})$ intersect in a unique point. Given simplices $|A|, |B|$ in $C(m, \delta)$, we say that $|A|$ and $|B|$ \emph{intersect transversely} if there is a circuit $(Z_{+}, Z_{-})$ of $C(m, \delta)$ such that $A \supseteq Z_{+}$ and $B \supseteq Z_{-}$.

The circuits of a cyclic polytope admit a combinatorial description. Following \cite[Definition 2.2]{ot}, if $A,B \subseteq [m]$ are $(d+1)$-tuples, then we say that $A$ \emph{intertwines} $B$, and write $A \wr B$, if \[a_{0}<b_{0}<a_{1}<b_{1}<\dots<a_{d}<b_{d}.\] If either $A \wr B$ or $B \wr A$, then we say that $A$ and $B$ are \emph{intertwining}. (That is, we use `are intertwining' to refer to the symmetric closure of the relation `intertwines'.)  A collection of increasing $(d+1)$-tuples is called \emph{non-intertwining} if no pair of the elements are intertwining. We also extend the terminology of \cite[Definition 2.2]{ot} to the case where $A$ is a $d$-tuple and $B$ is a $(d+1)$-tuple. Here we say that $A$ \emph{intertwines} $B$, and write $A \wr B$, if \[b_{0}<a_{0}<b_{1}<\dots<a_{d-1}<b_{d}.\] The circuits of $C(m,\delta)$ are then the pairs $(A,B)$ and $(B,A)$ such that $A$ is a $(\lfloor \frac{\delta}{2} \rfloor+1)$-tuple, $B$ is a $(\lceil \frac{\delta}{2} \rceil +1)$-tuple, and $A$ intertwines $B$. This result is originally due to Breen \cite{breen}, but also follows from the description of the oriented matroid given by a cyclic polytope \cite{b-lv,cd,sturmfels}.

A \emph{facet} of $C(m,\delta)$ is a face of codimension one. The \emph{upper facets} of $C(m,\delta)$ are those that can be seen from a very large positive $\delta$-th coordinate. The \emph{lower facets} of $C(m,\delta)$ are those that can be seen from a very large negative $\delta$-th coordinate.

Given a tuple $F \subset [m]$, we say that an element $v \in [m] \setminus F$ is an \emph{even gap} in $F$ if $\#\{ x \in F \mid x > v \}$ is even. Otherwise, it is an \emph{odd gap}. A subset $F \subset [m]$ is \emph{even} if every $v \in [m] \setminus F$ is an even gap. A subset $F \subset [m]$ is \emph{odd} if every $v \in [m] \setminus F$ is an odd gap. By Gale's Evenness Criterion \cite[Theorem 3]{gale}\cite[Lemma 2.3]{er}, given a $\delta$-tuple $F \subset [m]$, we have that $|F|$ is an upper facet if and only if $F$ is an odd subset, and that $|F|$ is a lower facet if and only if $F$ is an even subset.

The collection of lower facets $\{|F|_{\delta+1}\}$ of $C(m,\delta+1)$ gives a triangulation $\{|F|_{\delta}\}$ of $C(m,\delta)$, known as the \emph{lower triangulation}. Likewise, the collection of upper facets of $C(m,\delta+1)$ gives a triangulation of $C(m,\delta)$ known as the \emph{upper triangulation}.

Indeed, every triangulation $\mathcal{T}$ of $C(m, \delta)$ determines a unique piecewise-linear section \[s_{\mathcal{T}}\colon C(m,\delta)\rightarrow C(m,\delta+1)\] of $C(m, \delta +1)$ by sending each $\delta$-simplex $|A|_{\delta}$ of $\mathcal{T}$ to $|A|_{\delta+1}$ in $C(m, \delta+1)$, in the natural way. Similarly, a $\delta$-simplex $|A|$ in $C(m,\delta)$ defines a map $s_{|A|}\colon |A|_{\delta} \rightarrow C(m,\delta+1)$. The \emph{second higher Stasheff--Tamari order} on $\mathcal{S}(m,\delta)$ is defined as \[ \mathcal{T} \leqslant_{2} \mathcal{T}' \iff s_{\mathcal{T}}(x)_{\delta+1} \leqslant s_{\mathcal{T}'}(x)_{\delta+1} ~ \forall x \in C(m, \delta),\] where $s_{\mathcal{T}}(x)_{\delta+1}$ denotes the $(\delta+1)$-th coordinate of the point $s_{\mathcal{T}}(x)$. We write $\mathcal{S}_{2}(m,\delta)$ for the poset on $\mathcal{S}(m,\delta)$ this gives. Given triangulations $\mathcal{T}\in \mathcal{S}(m,\delta)$ and $\mathcal{T}' \in \mathcal{S}(m,\delta+1)$, we say that $\mathcal{T}$ is a \emph{section} of $\mathcal{T}'$ if $s_{\mathcal{T}}(C(m,\delta))$ is contained in $\mathcal{T}'$ as a simplicial subcomplex.

We also use the following different interpretation of the second higher Stasheff--Tamari order. A $k$-simplex $\Sigma$ in $C(m, \delta)$ is \emph{submerged} by the triangulation $\mathcal{T} \in \mathcal{S}(m, \delta)$ if the restriction of the piecewise linear section $s_{\mathcal{T}}$ to the simplex $\Sigma$ has the property that \[s_{\Sigma}(x)_{\delta+1} \leqslant s_{\mathcal{T}}(x)_{\delta+1}\] for all points $x \in \Sigma$. For a triangulation $\mathcal{T}$ of $C(m, \delta)$, the \emph{$k$-submersion set}, $\mathrm{sub}_{k}(\mathcal{T})$, is the set of $k$-simplices $\Sigma$ which are submerged by $\mathcal{T}$. Given two triangulations $\mathcal{T}, \mathcal{T}' \in \mathcal{S}(m,\delta)$, we have that $\mathcal{T} \leqslant_{2} \mathcal{T}'$ if and only if $\mathrm{sub}_{\lceil \frac{\delta}{2} \rceil}(\mathcal{T}) \subseteq \mathrm{sub}_{\lceil \frac{\delta}{2} \rceil}(\mathcal{T}')$ \cite[Proposition 2.15]{er}.

We now define the first higher Stasheff--Tamari order. Consider the cyclic polytope $C(\delta+2,\delta)$. Then any triangulation $\mathcal{T}$ of $C(\delta+2,\delta)$ determines a section $s_{\mathcal{T}}\colon C(\delta+2,\delta)\rightarrow C(\delta+2,\delta+1)$. But $C(\delta+2,\delta+1)$ is a simplex. It therefore has only one triangulation and only two sections: one corresponding to its upper facets and one corresponding to its lower facets. Hence the only two triangulations of $C(\delta+2,\delta)$ are the upper triangulation and the lower triangulation. For example, when $\delta=2$ the polytope $C(\delta+2,\delta)$ is a quadrilateral. This has two triangulations, corresponding to the two possible diagonals.

Let $\mathcal{T} \in \mathcal{S}(m,\delta)$. Suppose that there exists a $(\delta+2)$-tuple $X \subseteq [m]$ such that $\mathcal{T}$ restricts to the lower triangulation of $C(X,\delta)$. Let $\mathcal{T}'$ be the triangulation obtained by replacing the portion of $\mathcal{T}$ inside $C(X,\delta)$ with the upper triangulation of $C(X,\delta)$. We then say that $\mathcal{T}'$ is an \emph{increasing bistellar flip} of $\mathcal{T}$ and that $\mathcal{T}$ is a \emph{decreasing bistellar flip} of $\mathcal{T}'$.

The covering relations of the \emph{first higher Stasheff--Tamari order} are that $\mathcal{T} \lessdot_{1} \mathcal{T}'$ if and only if $\mathcal{T}'$ is an increasing bistellar flip of $\mathcal{T}$. We write $\mathcal{S}_{1}(m,\delta)$ for the poset on $\mathcal{S}(m,\delta)$ this gives and $\leqslant_{1}$ for the partial order itself. The first higher Stasheff--Tamari order was originally introduced by Kapranov and Voevodsky in \cite[Definition 3.3]{kv-poly} as the ``higher Stasheff order'' using a slightly different definition. Thomas showed in \cite[Proposition 3.3]{thomas-bst} that the higher Stasheff order of Kapranov and Voevodsky was the same as the first higher Stasheff--Tamari order of Edelman and Reiner.

General introductions to the higher Stasheff--Tamari orders can be found in \cite{rr} and \cite[Section 6.1]{lrs}.

\subsection{Higher Auslander--Reiten theory}\label{subsect-high-ar}

In this section let $\Lambda$ be a finite-dimensional algebra over a field $K$. We denote by $\mathrm{mod}\,\Lambda$ the category of finite-dimensional right $\Lambda$-modules. Given a module $M \in \mathrm{mod}\,\Lambda$, the subcategory $\mathrm{add}\,M$ consists of summands of direct sums of copies of $M$.

Given a subcategory $\mathcal{X} \subset \modules\Lambda$ and a map $f:X \rightarrow M$, where $X \in \mathcal{X}$ and $M \in \modules\Lambda$, we say that $f$ is a \emph{right $\mathcal{X}$-approximation} if for any $X' \in \mathcal{X}$, the sequence \[\mathrm{Hom}_{\Lambda}(X',X) \rightarrow \mathrm{Hom}_{\Lambda}(X',M)\rightarrow 0\] is exact, following \cite{as-preproj}. \emph{Left $\mathcal{X}$-approximations} are defined dually. 	The subcategory $\mathcal{X}$ is said to be \emph{contravariantly finite} if every $M \in \modules\Lambda$ admits a right $\mathcal{X}$-approximation, and \emph{covariantly finite} if every $M \in \modules\Lambda$ admits a left $\mathcal{X}$-approximation. If $\mathcal{X}$ is both contravariantly finite and covariantly finite, then $\mathcal{X}$ is \emph{functorially finite}.

Higher Auslander--Reiten theory was introduced by Iyama in \cite{iy-aus,iy-high-ar,iy-clus} as a higher-dimensional generalisation of classical Auslander--Reiten theory. For more detailed background to the theory, see the papers \cite{jk-intro,jasso,gko,io,iy-revis}.

The following subcategories provide the setting for the higher theory. Let $\mathcal{M}$ be a functorially finite subcategory of $\mathrm{mod}\,\Lambda$. Then we call $\mathcal{M}$ \emph{$d$-cluster-tilting} if
\begin{align*}
\mathcal{M} &= \{ X \in \modules\Lambda \mid \forall i \in [d-1], \forall M \in \mathcal{M}, \mathrm{Ext}_{\Lambda}^{i}(X,M) = 0 \} \\
&= \{ X \in \modules\Lambda \mid \forall i \in [d-1], \forall M \in \mathcal{M}, \mathrm{Ext}_{\Lambda}^{i}(M,X) = 0 \}.
\end{align*}

In the case $d=1$, the conditions should be interpreted as being trivial, so that $\mathrm{mod}\,\Lambda$ is the unique $1$-cluster-tilting subcategory of $\mathrm{mod}\,\Lambda$. If $\mathrm{add}\,M$ is a $d$-cluster-tilting subcategory, for $M \in \mathrm{mod}\,\Lambda$, then we say that $M$ is a \emph{$d$-cluster-tilting module}.

We say that $\Lambda$ is \emph{weakly $d$-representation-finite} if there exists a $d$-cluster-tilting module in $\mathrm{mod}\,\Lambda$, following \cite[Definition 2.2]{io}. If, additionally, $\mathrm{gl.dim}\,\Lambda \leqslant d$, we say that $\Lambda$ is \emph{$d$-representation-finite $d$-hereditary}, following \cite[Definition 1.25]{jk-nak} and \cite{hio}. (These latter algebras are simply called ``$d$-representation-finite'' in \cite{io}.)

The canonical examples of $d$-representation-finite $d$-hereditary algebras are the higher Auslander algebras of linearly oriented $A_{n}$, introduced by Iyama in \cite{iy-clus}. The construction we give here is based on \cite[Construction 3.4]{ot} and \cite[Definition 5.1]{io}.

Following \cite{ot}, we denote the sets
\begin{align*}
\modset{m}{d} &:= \{(a_{0}, \dots , a_{d}) \in [m]^{d+1} \mid \forall i \in \{ 0, 1, \dots ,d-1 \}, a_{i+1}\geqslant a_{i}+2 \},\\
\nonconsec{m}{d} &:= \{(a_{0}, \dots , a_{d}) \in \modset{m}{d} \mid a_{d}+ 2 \leqslant a_{0} + m \}.
\end{align*}
We say that a $(d+1)$-tuple $A$ is \emph{separated} if $A \in \modset{m}{d}$ for some $m$.

Let $Q^{(d,n)}$ be the quiver with vertices \[Q_{0}^{(d,n)} := \modset{n+2d-2}{d-1}\] and arrows \[Q_{1}^{(d,n)} := \{A \rightarrow A+1_{i} \mid A, A+1_{i} \in Q_{0}^{(d,n)}\},\] where \[1_{i}:= (0, \dots, 0, \overset{i}{1}, 0 , \dots , 0).\] We multiply arrows as if they were functions, so that $\xrightarrow{\alpha}\xrightarrow{\beta}=\beta\alpha$.

Let $A_{n}^{d}$ be the quotient of the path algebra $KQ^{(d,n)}$ by the relations:
\begin{equation*}
A \rightarrow A+1_{i} \rightarrow A+1_{i}+1_{j} = 
\left\{
	\begin{array}{ll}
		A \rightarrow A+1_{j} \rightarrow A+1_{i}+1_{j} & \quad \text{if } A+1_{j} \in Q_{0}^{(d,n)} \\
		0 & \quad \text{otherwise.}
	\end{array}
\right.
\end{equation*}

It is shown in \cite{iy-clus} that the algebra $A_{n}^{d}$ is $d$-representation-finite $d$-hereditary with unique basic $d$-cluster-tilting module $M^{(d,n)}$ and that \[ A_{n}^{d+1} \cong \mathrm{End}_{A_{n}^{d}}M^{(d,n)}.\]

\cite{ot} gives two different algebraic frameworks for triangulations of cyclic polytopes, one using tilting \cite[Section 3, Section 4]{ot} and the other using cluster-tilting \cite[Section 5, Section 6]{ot}. Interesting results arise in both frameworks, so we choose to consider both. For example, the tilting framework reveals connections with the orders on tilting modules studied in \cite{rs-simp,hu-potm}; and the cluster-tilting framework relates to maximal green sequences. For consistency, we primarily prove results in the tilting framework. In Sections \ref{sect-clus-tilt} and \ref{sect-max-green} we indicate the analogues of these results in the cluster-tilting framework. We now explain the two different settings.

\subsubsection{Tilting}\label{sect-back-tilt}

Tilting modules of projective dimension one were defined by Brenner and Butler \cite{bb} as a generalisation of BGP reflection functors \cite{bgp,apr}. Miyashita defined tilting modules of higher projective dimension \cite{miya-tilt}. This was in turn generalised by Cline, Parshall, and Scott, whose definition we use here \cite[Definition 2.3]{cps}. Given a $\Lambda$-module $T$, we say that $T$ is a \emph{tilting} module if:
\begin{enumerate}
\item the projective dimension of $T$ is finite;
\item $\mathrm{Ext}_{\Lambda}^{i}(T,T)=0$ for all $i>0$; that is, $T$ is \emph{rigid};
\item there is an exact sequence $0 \rightarrow \Lambda \rightarrow T_{0} \rightarrow \dots \rightarrow T_{s} \rightarrow 0$ with each $T_{i} \in \mathrm{add}\,T$.
\end{enumerate}

Given a $d$-representation-finite $d$-hereditary algebra $\Lambda$ with $d$-cluster-tilting module $M$ and tilting modules $T, T' \in \mathrm{add}\,M$, we say that $T'$ is a \emph{left mutation} of $T$ if and only if $T=E \oplus X, T'=E \oplus Y$ and there is an exact sequence \[0 \rightarrow X \rightarrow E_{1} \rightarrow \dots \rightarrow E_{d} \rightarrow Y \rightarrow 0\] where $X$ and $Y$ are indecomposable and $E_{i} \in \mathrm{add}\,E$.

Given a tilting module $T \in \mathrm{add}\,M^{(d,n)}$, we denote \[\prescript{\bot}{}T := \{ X \in \mathrm{add}\,M^{(d,n)} \mid \mathrm{Ext}_{A_{n}^{d}}^{i}(X,T)=0 ~ \forall i >0\}.\] The set $T^{\bot}$ is defined dually. Since $\mathrm{gl.dim}\,A_{n}^{d} \leqslant d$ and $\mathrm{add}\,M^{(d,n)}$ is a $d$-cluster-tilting subcategory, we have that \[\prescript{\bot}{}T = \{ X \in \mathrm{add}\,M^{(d,n)} \mid \mathrm{Ext}_{A_{n}^{d}}^{d}(X,T)=0\}.\]

By \cite[Theorem 1.1]{ot}, there are bijections between
\begin{itemize}
\item elements of $\modset{n+2d}{d}$,
\item $d$-simplices of $C(n+2d,2d)$ which do not lie in any lower facet, and
\item indecomposable direct summands of $M^{(d,n)}$.
\end{itemize}
These induce bijections between
\begin{itemize}
\item non-intertwining collections of $\binom{n+d-1}{d}$ $(d+1)$-tuples in $\modset{n+2d}{d}$,
\item triangulations of $C(n+2d,2d)$, and
\item basic tilting $A_{n}^{d}$-modules contained in $M^{(d,n)}$,
\end{itemize}
and also bijections between
\begin{itemize}
\item elements of $\nonconsec{n+2d}{2d}$,
\item internal $d$-simplices of $C(n+2d,2d)$, and
\item indecomposable non-projective-injective direct summands of $M^{(d,n)}$.
\end{itemize}
Here an \emph{internal} simplex of $C(m,2d)$ is one that does not lie in any facet of the polytope.

Given a triangulation $\mathcal{T} \in \mathcal{S}(n+2d,2d)$, we write $e(\mathcal{T})$ for the corresponding set of non-intertwining $(d+1)$-tuples in $\modset{n+2d}{d}$, following \cite{ot}. Given an element $A \in \modset{n+2d}{d}$, we use $M_{A}$ to refer to the corresponding indecomposable summand of $M^{(d,n)}$. The following proposition is key to our algebraic characterisation of the higher Stasheff--Tamari orders in even dimensions.

\begin{proposition}\cite[Theorem 3.6(4)]{ot}\label{prop-key}
We have that $\mathrm{Ext}_{A_{n}^{d}}^{d}(M_{B}, M_{A}) \neq 0$ if and only if $A \wr B$.
\end{proposition}

\subsubsection{Cluster-tilting}\label{sect-back-clus-tilt}

Our set-up for cluster-tilting is slightly different from that of \cite{ot} because we wish to consider orders on triangulations. Hence we consider a subcategory of the derived category rather than the higher cluster category $\mathcal{O}_{A_{n}^{d}}$ they define. This follows the practice of other authors, such as \cite{hi-no-gap}.

Let $\Lambda$ be a $d$-representation-finite $d$-hereditary algebra with $d$-cluster-tilting module $M$. Let $\mathcal{D}_{\Lambda}:=D^{b}(\mathrm{mod}\,\Lambda)$ be the bounded derived category of finitely generated $\Lambda$-modules. We denote the shift functor in the derived category by $[1]$ and its $d$-th power by $[d]:=[1]^{d}$. Let \[\mathcal{U}_{\Lambda}:=\mathrm{add}\,\{M[id] \in \mathcal{D}_{\Lambda} \mid i \in \mathbb{Z}\}\] be a subcategory of $\mathcal{D}_{\Lambda}$. This is a $(d+2)$-angulated category in the sense of \cite{gko}. Consider the subcategory $\mathcal{V}_{\Lambda}:=\mathrm{add}\,(M \oplus \Lambda[d])$ of $\mathcal{U}_{\Lambda}$. We say that $T \in \mathcal{V}_{\Lambda}$ is a \emph{(basic) cluster-tilting object} if $\mathrm{Hom}_{\mathcal{U}_{\Lambda}}(T, T[d]) = 0$ and if $T$ has $m$ indecomposable summands which are pairwise non-isomorphic, where $m$ is the number of indecomposable summands of $\Lambda$ as a $\Lambda$-module. Note that both $\Lambda$ and $\Lambda[d]$ are cluster-tilting objects.

Given two cluster-tilting objects $T,T' \in \mathcal{V}_{\Lambda}$, we say that $T'$ is a \emph{left mutation} of $T$ if $T=E \oplus X$, $T'=E \oplus Y$ and there exists a $(d+2)$-angle \[X \rightarrow E_{1} \rightarrow \dots \rightarrow E_{d} \rightarrow Y \rightarrow X[d].\]

As before, given a cluster-tilting object $T \in \mathcal{V}_{A_{n}^{d}}$, we denote 
\begin{align*}
\prescript{\bot}{}T &:= \{ X \in \mathcal{V}_{A_{n}^{d}} \mid \mathrm{Hom}_{\mathcal{D}_{A_{n}^{d}}}(X,T[i])=0 ~ \forall i >0\}\\
&=\{ X \in \mathcal{V}_{A_{n}^{d}} \mid \mathrm{Hom}_{\mathcal{D}_{A_{n}^{d}}}(X,T[d])=0\}.
\end{align*}

\begin{remark}\label{rmk-clus-tilt}
The indecomposable objects of $\mathcal{V}_{\Lambda}$ are in bijection with indecomposable objects of the cluster category $\mathcal{O}_{\Lambda}$ defined in \cite[Definition 5.22]{ot} by \cite[Theorem 5.2(1)]{ot}. It follows from \cite[Theorem 3.5(i)]{jj-tau} that our cluster-tilting objects correspond to rigid objects in $\mathcal{O}_{\Lambda}$ with $m$ summands.

However, it is not clear for general $\Lambda$ that our cluster-tilting objects always correspond to cluster-tilting objects in this cluster category, as defined in \cite[Definition 5.3]{ot}. We call these \emph{Oppermann--Thomas cluster-tilting objects}, following \cite{jj-tau}. This is because it is not clear that the images of our cluster-tilting objects in $\mathcal{O}_{\Lambda}$ will always satisfy \cite[Definition 5.3(2)]{ot}. But it is true that Oppermann--Thomas cluster-tilting objects always give cluster-tilting objects in our sense, because they must always have $m$ summands, as shown in \cite{reid-trop}. Note that for the algebras $A_{n}^{d}$, however, our notion of cluster-tilting coincides exactly with the notion of Oppermann--Thomas cluster-tilting, since by \cite[Theorem 1.1, Theorem 1.2]{ot} Oppermann--Thomas cluster-tilting objects for $A_{n}^{d}$ are precisely rigid objects with the maximum number of non-isomorphic indecomposable summands.
\end{remark}

By \cite[Theorem 1.1, Theorem 1.2]{ot} there are bijections between:
\begin{itemize}
\item elements of $\nonconsec{n+2d+1}{d}$,
\item internal $d$-simplices of $C(n+2d+1,2d)$,
\item indecomposable objects in $\mathcal{V}_{A_{n}^{d}}$, and
\item indecomposable non-projective-injective direct summands of $M^{(d,n+1)}$.
\end{itemize}
These induce bijections between:
\begin{itemize}
\item non-intertwining collections of $\binom{n+d-1}{d}$ $(d+1)$-tuples in $\nonconsec{n+2d+1}{d}$,
\item triangulations of $C(n+2d+1,2d)$,
\item basic cluster-tilting objects in $\mathcal{V}_{A_{n}^{d}}$, and
\item basic tilting $A_{n+1}^{d}$-modules contained in $M^{(d,n+1)}$.
\end{itemize}
Moreover, the following analogue of Proposition \ref{prop-key} holds in the cluster-tilting framework. Hence, we may translate freely between the tilting and cluster-tilting settings.

\begin{proposition}\cite[Proof of Proposition 6.1]{ot}
We have that $\mathrm{Hom}_{\mathcal{U}_{A_{n}^{d}}}(M_{B}, M_{A}[d]) \neq 0$ if and only if $A \wr B$.
\end{proposition}

\section{The higher Stasheff--Tamari orders in even dimensions}\label{sect-hst}

In this section, we give new combinatorial interpretations of the higher Stasheff--Tamari orders on triangulations of even-dimensional cyclic polytopes using the framework for triangulations from \cite{ot}. We then show how these orders translate to orders on tilting $A_{n}^{d}$-modules, giving the algebraic interpretation. We finally note how one can equally interpret the orders algebraically using cluster-tilting objects for $A_{n}^{d}$ instead.

\begin{remark}
The higher Stasheff--Tamari orders behave differently in even and odd dimensions. In even dimensions, both posets are self-dual, whereas neither is self-dual in odd dimensions; likewise, in odd dimensions both posets possess an involution which does not exist in even dimensions \cite[Proposition 2.11]{er}. The first order is a ranked poset in odd dimensions, but neither order is ranked in even dimensions \cite[Corollary 1.2]{rambau}. These properties indicate that the orders are suited to different combinatorial interpretations in odd and even dimensions.
\end{remark}

The following lemma is key to proving our combinatorial interpretations of the higher Stasheff--Tamari orders in even dimensions.

\begin{lemma}\label{up-low-fac}
Let $|A|_{2d}$ and $|B|_{2d}$ be $d$-simplices in $C(m,2d)$. Then $A \wr B$ if and only if we have that $|A \cup B|_{2d+1}$ is a $(2d+1)$-simplex with $|A|_{2d+1}$ the intersection of its lower facets, and $|B|_{2d+1}$ the intersection of its upper facets.
\end{lemma}
\begin{proof}
Let $A \cup B =: S$. If $\# S < 2d+2$, then $A$ and $B$ cannot be intertwining. Hence, suppose that $\# S=2d+2$, so that $|A \cup B|_{2d+1}$ is a $(2d+1)$-simplex. Then, by Gale's Evenness Criterion, the intersection of the lower facets of $S$ is $|s_{0}, s_{2}, \dots, s_{2d}|_{2d+1}$ and the intersection of the upper facets is $|s_{1}, s_{3}, \dots, s_{2d+1}|_{2d+1}$.
\end{proof}

\subsection{First order}\label{even-first}

We start by showing our combinatorial interpretation of the first higher Stasheff--Tamari order.

By \cite[Theorem 4.1]{ot}, triangulations $\mathcal{T},\mathcal{T}' \in \mathcal{S}(m, 2d)$ are bistellar flips of each other if and only if $e(\mathcal{T)}$ and $e(\mathcal{T}')$ have all but one $(d+1)$-tuple in common. This can then be strengthened to the following.

\begin{theorem}\label{main-thm-1st}
For $\mathcal{T}, \mathcal{T}' \in \mathcal{S}(m,2d)$, we have that $\mathcal{T} \lessdot_{1} \mathcal{T}'$ if and only if $e(\mathcal{T}) = \mathcal{U} \cup \{ A\}$  and $e(\mathcal{T}')= \mathcal{U} \cup \{ B\}$  and $A \wr B$.
\end{theorem}
\begin{proof}
Consider a bistellar flip between $\mathcal{T}$ and $\mathcal{T}'$ inside $C(S,2d)$, where $S \in \subs{m}{2d+2}{}$. By Lemma \ref{up-low-fac}, the only $d$-simplex contained in upper facets of $|S|_{2d+1}$ but not any lower facets is $|s_{1}, s_{3}, \dots, s_{2d+1}|_{2d+1}$. Similarly the only $d$-simplex contained in lower facets but not any upper facets is $|s_{0},s_{2}, \dots, s_{2d}|_{2d+1}$. Moreover, these are both internal $d$-simplices in $C(m,2d)$, since $(s_{1}, s_{3}, \dots, s_{2d+1}), (s_{0},s_{2}, \dots, s_{2d}) \in \nonconsec{m}{d}$. Hence an increasing bistellar flip inside $C(S, 2d)$ involves exchanging $(s_{0}, \dots, s_{2d})$ for $(s_{1}, \dots, s_{2d+1})$. Therefore, if $\mathcal{T}'$ is an increasing bistellar flip of $\mathcal{T}$, we have that $e(\mathcal{T}) = \mathcal{U} \cup \{ A\}$  and $e(\mathcal{T}')= \mathcal{U} \cup \{ B\}$  and $A \wr B$.

Conversely, suppose that we have $e(\mathcal{T}) = \mathcal{U} \cup \{ A\}$  and $e(\mathcal{T}')= \mathcal{U} \cup \{ B\}$  and $A \wr B$. By \cite[Theorem 4.1]{ot} and its proof, we have that $\mathcal{T}'$ is the result of a bistellar flip of $\mathcal{T}$ which takes place inside $C(A \cup B, 2d)$. Then, by the previous paragraph, this must be an increasing bistellar flip.
\end{proof}

This theorem can be interpreted algebraically.

\begin{theorem}\label{alg-cor-1st}
Let $\mathcal{T}, \mathcal{T}' \in \mathcal{S}(n+2d,2d)$ with corresponding tilting $A_{n}^{d}$-modules $T$ and $T'$. Then $\mathcal{T} \lessdot_{1} \mathcal{T}'$ if and only if $T'$ is a left mutation of $T$.
\end{theorem}
\begin{proof}
Proposition \ref{prop-key} and \cite[Theorem 3.8]{ot} imply that $T'$ is a left mutation of $T$ if and only if $T=E \oplus M_{A}$, $T'=E \oplus M_{B}$ and $A \wr B$. By Theorem \ref{main-thm-1st}, this is true if and only if $\mathcal{T} \lessdot_{1} \mathcal{T}'$.
\end{proof}

\begin{example}\label{ex-even-first}
We illustrate the results of this section with an example. Consider the cyclic polytope $C(6,4)$. This has two triangulations, $\mathcal{T}$ and $\mathcal{T}'$. We have that
\begin{align*}
e(\mathcal{T})&=\{135,136,146\}, \\
e(\mathcal{T}')&=\{136,146,246\}.
\end{align*}
To keep the notation light, we omit brackets and commas from tuples, so that $135=(1,3,5)$. Hence Theorem \ref{main-thm-1st} shows us that $\mathcal{T}'$ is an increasing bistellar flip of $\mathcal{T}$, since $e(\mathcal{T})$ and $e(\mathcal{T}')$ have all but one 3-tuple in common, and $135 \wr 246$.

The algebra $A_{2}^{2}$ is\[
\begin{tikzcd}
& 14 \ar[dr] & \\
13 \ar[ur] \ar[rr,dash,dashed] && 24,
\end{tikzcd}\]which we henceforth abbreviate\[
\begin{tikzcd}
& 2 \ar[dr] & \\
1 \ar[ur] \ar[rr,dash,dashed] && 3.
\end{tikzcd}\]
The 2-cluster-tilting subcategory of $\modules A_{2}^{2}$ is\[
\begin{tikzcd}
& {\tcs{2\\1}} \ar[dr] & \\
{\tcs{1}} \ar[ur] && {\tcs{3\\2}} \ar[dl] \\
& {\tcs{3}}. &
\end{tikzcd}\]
The corresponding internal $d$-simplices of $C(6,4)$ are given by\[
\begin{tikzcd}
& M_{136} \ar[dr] & \\
M_{135} \ar[ur] && M_{146} \ar[dl] \\
& M_{246}. &
\end{tikzcd}\]
Then $\mathcal{T}$ and $\mathcal{T}'$ correspond to the respective tilting modules
\begin{align*}
T&:={\tcs{1}} \oplus {\tcs{2\\1}} \oplus {\tcs{3\\2}}\,, \\
T'&:={\tcs{2\\1}} \oplus {\tcs{3\\2}} \oplus {\tcs{3}}.
\end{align*}
We have that $T'$ is a left mutation of $T$ via the exact sequence \[0 \rightarrow {\tcs{1}} \rightarrow {\tcs{2\\1}} \rightarrow {\tcs{3\\2}} \rightarrow {\tcs{3}} \rightarrow 0.\]
\end{example}

\subsection{Second order}\label{even-second}

In this section we prove our combinatorial and algebraic interpretations of the second order in even dimensions. We start with the combinatorial interpretation of the second order, which is as follows.

\begin{theorem}\label{main-thm-2nd}
Let $\mathcal{T}, \mathcal{T}' \in \mathcal{S}(m,2d)$. Then $\mathcal{T} \leqslant_{2} \mathcal{T}'$ if and only if for every $A \in e(\mathcal{T})$, there is no $B \in e(\mathcal{T}')$ such that $B \wr A$.
\end{theorem}

To prove this theorem we use the following combinatorial characterisation of submersion in even dimensions. This was shown for $d=1$ in \cite[Proposition 3.2]{er}.

\begin{proposition}\label{desc-sub}
Let $\mathcal{T} \in \mathcal{S}(m,2d)$. Let $|A|$ be an internal $d$-simplex in $C(m,2d)$. Then $|A|$ is submerged by $\mathcal{T}$ if and only if there is no $B \in e(\mathcal{T})$ such that $B \wr A$. 
\end{proposition}
\begin{proof}
We prove the backwards direction first. Suppose that $|A|$ is not submerged by $\mathcal{T}$, so that there is a point $x$ in $|A|$ such that $s_{|A|}(x)_{2d+1} > s_{\mathcal{T}}(x)_{2d+1}$. We split into two cases, depending on whether $s_{|A|}(y)_{2d+1} > s_{\mathcal{T}}(y)_{2d+1}$ for all $y \in \mathring{|A|}$, or whether there are also some $y \in \mathring{|A|}$ such that $s_{A}(y)_{2d+1} \leqslant s_{\mathcal{T}}(y)_{2d+1}$.

In the first case, there must exist $B \in e(\mathcal{T})$ such that $A$ and $B$ are intertwining, since $A \notin e(\mathcal{T})$. By Lemma \ref{up-low-fac}, $|A|_{2d+1}$ must the intersection of either the lower facets or the upper facets of the $(2d+1)$-simplex $|A \cup B|_{2d+1}$. We have that $\mathring{|A|}$ and $\mathring{|B|}$ intersect in a unique point $y \in C(m,2d)$, since $(A,B)$ is a circuit. By assumption, $s_{|A|}(y)_{2d+1} > s_{\mathcal{T}}(y)_{2d+1}=s_{|B|}(y)_{2d+1}$. This means that $|A|_{2d+1}$ must be the intersection of the upper facets of $|A \cup B|_{2d+1}$. Hence $B \wr A$ by Lemma \ref{up-low-fac}.

In the second case, by continuity, we must have a point $z$ in $\mathring{|A|}$ such that $s_{|A|}(z)_{2d+1} = s_{\mathcal{T}}(z)_{2d+1}$. Hence, $z$ is a point of intersection between $|A|_{2d+1}$ and the image of $\mathcal{T}$ under $s_{\mathcal{T}}$. The point $z$ must be contained in a $2d$-simplex $|S|_{2d+1}$ of $\mathcal{T}$. Then, by the description of the circuits in $C(m,2d+1)$, there must exist $s'_{i} \in \{ s_{0}, \dots, s_{2d}\}$ such that \[s'_{0}<a_{0}<s'_{1}<a_{1}< \dots < s'_{d}<a_{d}< s'_{d+1}.\] Then $B = (s'_{0}, \dots, s'_{d})$ is a $(d+1)$-tuple of $e(\mathcal{T})$ such that $B \wr A$.

Now we prove the forwards direction by contraposition. If there is a $B \in e(\mathcal{T})$ such that $B \wr A$, then $|A|_{2d+1}$ is the intersection of the upper facets of the $(2d+1)$-simplex $|A \cup B|_{2d+1}$ and $|B|_{2d+1}$ is the intersection of the lower facets, by Lemma \ref{up-low-fac}. We have that $\mathring{|A|}_{2d}$ and $\mathring{|B|}_{2d}$ intersect in a unique point $x$. Then $s_{|B|}(x)_{2d+1} = s_{\mathcal{T}}(x)_{2d+1} < s_{|A|}(x)_{2d+1}$. But this means that $|A|$ is not submerged by $\mathcal{T}$.
\end{proof}

The following lemma shows that, in order to have $\mathcal{T} \leqslant_{2} \mathcal{T}'$, it is sufficient for $\mathcal{T}'$ to submerge the $d$-simplices of $\mathcal{T}$.

\begin{lemma}\label{lem-sub}
Let $\mathcal{T},\mathcal{T}' \in \mathcal{S}(m,2d)$. Then $\mathcal{T}\leqslant_{2} \mathcal{T}'$ if and only if every $d$-simplex of $\mathcal{T}$ is submerged by $\mathcal{T}'$.
\end{lemma}
\begin{proof}
The forwards direction is clear. Conversely, suppose that every $d$-simplex of $\mathcal{T}$ is submerged by $\mathcal{T}'$. Since every point in $C(m,2d)$ lies in a $2d$-simplex of $\mathcal{T}$, it suffices to show that every $2d$-simplex of $\mathcal{T}$ is submerged by $\mathcal{T}'$. Suppose that there is a $2d$-simplex $|A|$ of $\mathcal{T}$ such that $|A|$ is not submerged by $\mathcal{T}'$. We can assume that the $2d$-simplex $|A|$ has at least one face which is an internal $d$-simplex not belonging to $\mathcal{T}'$. Otherwise, $|A|$ is a $2d$-simplex of $\mathcal{T}'$ by \cite[Lemma 2.15]{ot}.

Hence let $|A'|$ be a $d$-face of $|A|$ which is an internal $d$-simplex not belonging to $\mathcal{T}'$. Then, since $|A'|$ is submerged by $\mathcal{T}'$, there exist points $x \in |A'| \subset |A|$ such that $s_{|A'|}(x)_{2d+1}=s_{|A|}(x)_{2d+1}<s_{\mathcal{T}'}(x)$. Since $|A|$ is not submerged by $\mathcal{T}'$, there must also be $y \in |A|$ such that $s_{|A|}(y)_{2d+1}>s_{\mathcal{T}'}(y)$. Therefore $|A|_{2d+1}$ intersects $s_{\mathcal{T}'}(C(m,2d))$ by continuity. By the description of the circuits of $C(m,2d+1)$, there must either be a $d$-face $|F|$ of $|A|$ and a $(d+1)$-simplex $|G|$ of $\mathcal{T}'$ such that $F \wr G$, or a $(d+1)$-face $|G|$ of $|A|$ and a $d$-simplex $|F|$ of $\mathcal{T}'$ such that $F \wr G$. In the first case, $(g_{0}, \dots, g_{d}) \wr F$, so that $F$ is not submerged by $\mathcal{T}'$ by Proposition \ref{desc-sub}, a contradiction. In the second case, $(g_{1}, \dots, g_{d+1})$ is a $d$-face of $|A|$ which is not submerged by $\mathcal{T}'$ by Proposition \ref{desc-sub}, since $F \wr (g_{1}, \dots, g_{d+1})$, which is also a contradiction. 
\end{proof}

Proposition \ref{desc-sub} and Lemma \ref{lem-sub} together prove Theorem \ref{main-thm-2nd}, since boundary $d$-simplices are in, and hence submerged by, every triangulation. We now interpret this theorem algebraically. The following is entailed by Proposition \ref{desc-sub} and Proposition \ref{prop-key}.

\begin{corollary}\label{sub<->perp}
Let $\mathcal{T} \in \mathcal{S}(n+2d,2d)$ with corresponding tilting $A_{n}^{d}$-module $T$. Let $|A|$ be an internal $d$-simplex in $C(n+2d,2d)$ with corresponding indecomposable $A_{n}^{d}$-module $M_{A}$. Then $|A|$ is submerged by $\mathcal{T}$ if and only if $M_{A} \in \prescript{\bot}{}T$.
\end{corollary}

\begin{theorem}\label{alg-cor-2nd}
Let $\mathcal{T}, \mathcal{T}' \in \mathcal{S}(n+2d,2d)$ with corresponding tilting $A_{n}^{d}$-modules $T$ and $T'$. Then $\mathcal{T} \leqslant_{2} \mathcal{T}'$ if and only if $\prescript{\bot}{}T \subseteq \prescript{\bot}{}T'$.
\end{theorem}
\begin{proof}
We use the interpretation of $\mathcal{S}_{2}(n+2d,2d)$ in terms of $d$-submersion sets and apply Corollary \ref{sub<->perp}.

Suppose that $\prescript{\bot}{}T \subseteq \prescript{\bot}{}T'$. Let $|A|$ be a $d$-simplex of $C(n+2d,2d)$. Since $d$-simplices which are on the boundary of $C(n+2d,2d)$ are in, and hence submerged by, every triangulation, we suppose that $|A|$ is internal. Suppose that $|A|$ is submerged by $\mathcal{T}$. Then, by Corollary \ref{sub<->perp}, $M_{A} \in \prescript{\bot}{}T \subseteq \prescript{\bot}{}T'$. Therefore $|A|$ is submerged by $\mathcal{T}'$, and so $\mathcal{T} \leqslant_{2} \mathcal{T}'$.

Suppose that $\mathcal{T} \leqslant_{2} \mathcal{T}'$. Let $M_{A} \in \prescript{\bot}{}T$ be an indecomposable $A_{n}^{d}$-module corresponding to a $d$-simplex $|A|$ in $C(n+2d,2d)$. If $M_{A}$ is projective-injective, then it follows that $M_{A} \in \prescript{\bot}{}T'$ immediately. Hence, we suppose that $M_{A}$ is not projective-injective, so that $|A|$ is an internal $d$-simplex of $C(n+2d,2d)$. It follows by Corollary \ref{sub<->perp} that $|A|$ is submerged by $\mathcal{T}$, and so by $\mathcal{T}'$ as well. Then we have that $M_{A} \in \prescript{\bot}{}T'$, so $\prescript{\bot}{}T \subseteq \prescript{\bot}{}T'$.
\end{proof}

\begin{remark}
Theorems \ref{alg-cor-1st} and \ref{alg-cor-2nd} open up the Edelman--Reiner conjecture to the techniques of homological algebra. The orders in these theorems are higher-dimensional versions of the orders on tilting modules introduced in \cite{rs-simp}. These were shown to have the same Hasse diagram in \cite{hu-potm}, which implies that the two classical orders are equal for a representation-finite algebra. A higher dimensional version of this result would entail the equivalence of the higher Stasheff--Tamari orders in even dimensions.

Indeed, since \cite{hu-potm} shows that the two orders on tilting modules have the same Hasse diagram for $d=1$, this provides a neat algebraic proof of the result that the two higher Stasheff--Tamari orders are equal in dimension two \cite[Theorem 3.8]{er}.
\end{remark}

\begin{example}\label{ex-even-second}
We illustrate how the second higher Stasheff--Tamari order figures in Example \ref{ex-even-first}. Recall that we considered the cyclic polytope $C(6,4)$, which has two triangulations $\mathcal{T}$ and $\mathcal{T}'$, where $e(\mathcal{T})= \{135,136,146\}$ and
$e(\mathcal{T}') = \{136,146,246\}$. These correspond respectively to the tilting $A_{2}^{2}$-modules $T={\tcs{1}} \oplus {\tcs{2\\1}} \oplus {\tcs{3\\2}}$ and $T'={\tcs{2\\1}} \oplus {\tcs{3\\2}} \oplus {\tcs{3}}$.

We have that $\mathcal{T} \leqslant_{2} \mathcal{T}'$, since there are no tuples $B \in e(\mathcal{T}')$ and $A \in e(\mathcal{T})$ such that $B \wr A$. The only pair of tuples respectively from $e(\mathcal{T})$ and $e(\mathcal{T}')$ which are intertwining are $135$ and $246$, but we have $135 \wr 246$.

On the algebraic side, we have that $\prescript{\bot}{}T=\mathrm{add}\,\left\lbrace {\tcs{1}}\, ,\, {\tcs{2\\1}}\, ,\, {\tcs{3\\2}} \right\rbrace$ and $\prescript{\bot}{}T'=\mathrm{add}\,\left\lbrace {\tcs{1}}\, ,\,{\tcs{2\\1}}\, ,\, {\tcs{3\\2}}\, ,\, {\tcs{3}} \right\rbrace$, so that $\prescript{\bot}{}T \subseteq \prescript{\bot}{}T'$.
\end{example}

\subsection{Cluster-tilting}\label{sect-clus-tilt}

By similar arguments, the combinatorial characterisations of the higher Stasheff--Tamari orders given in Theorem \ref{main-thm-1st} and Theorem \ref{main-thm-2nd} translate into the following orders on cluster-tilting objects.

\begin{theorem}\label{thm-clus-tilt}
Let $\mathcal{T}, \mathcal{T}' \in \mathcal{S}(n+2d+1,2d)$ correspond to cluster-tilting objects $T,T' \in \mathcal{V}_{A_{n}^{d}}$.
\begin{enumerate}
\item We have that $\mathcal{T} \lessdot_{1} \mathcal{T}'$ if and only if $T'$ is a left mutation of $T$.
\item We have that $\mathcal{T} \leqslant_{2} \mathcal{T}'$ if and only if $\prescript{\bot}{}T \subseteq \prescript{\bot}{}T'$.
\end{enumerate}
\end{theorem}

\section{Triangulations of odd-dimensional cyclic polytopes}\label{sect-odd-desc}

In this section we characterise triangulations of $(2d+1)$-dimensional cyclic polytopes in terms of their $d$-simplices. This gives the odd-dimensional counterpart to the description in \cite{ot} of triangulations of even-dimensional cyclic polytopes as sets of non-intersecting $d$-simplices of a maximal size. It is far from obvious what the counterpart description for odd dimensions should look like. In a $(2d+1)$-dimensional cyclic polytope, $d$-simplices do not intersect each other, and numbers of simplices vary between triangulations, so that there is no notion of maximal size.

One appealing way of solving these problems might be to describe triangulations of $(2d+1)$-dimensional cyclic polytopes as inclusion-maximal sets of non-intersecting $d$-simplices and $(d+1)$-simplices. This is an approach based on the even-dimensional model. However, there are two issues here. The first issue is that a non-intersecting collection of simplices which is maximal with respect to adding more simplices does not necessarily give a triangulation of a cyclic polytope, as first shown in \cite[Example 4.5]{rambau}. This is why maximality of size is required in even dimensions, rather than simply maximality with respect to inclusion. But, as discussed above, numbers of simplices vary between triangulations in odd dimensions. The second issue is that, by \cite{dey}, a triangulation of a $(2d+1)$-dimensional cyclic polytope is determined by its $d$-simplices, so that including the $(d+1)$-simplices in a description is redundant.

We solve these problems by taking an approach which is distinctive to odd dimensions, rather than using the even-dimensional model. We describe triangulations of $(2d+1)$-dimensional cyclic polytopes in terms of their $d$-simplices by defining two new properties which imply that a given set of $d$-simplices arises from a triangulation. Indeed, we prove the following theorem, which we build up to using a series of lemmas. We extend our notation by using $\simp(\mathcal{T})$ to denote the set of tuples corresponding to internal $d$-simplices of $\mathcal{T}$, where $\mathcal{T}$ is a triangulation of $C(m,2d+1)$. We define the set \[\chainset{m}{d}:=\{(a_{0}, \dots, a_{d}) \in \modset{m}{d} \mid a_{0} \neq 1,\, a_{d} \neq m\}.\]

\begin{theorem}\label{thm-class-odd-dim}
There is a bijection via $\mathcal{T} \mapsto \simp(\mathcal{T})$ between triangulations of $C(m,2d+1)$ and subcollections of $\chainset{m}{d}$ which are supporting and bridging.
\end{theorem}

We begin by characterising the internal $d$-simplices of $C(m,2d+1)$.

\begin{lemma}\label{lem-int-odd}
Let $A\in \subs{m}{d+1}{}$. Then $|A|$ is an internal $d$-simplex of $C(m,2d+1)$ if and only if $A \in \chainset{m}{d}$.
\end{lemma}
\begin{proof}
One way to see this is to note that $|A|$ is an internal $d$-simplex if and only if there is a $(d+1)$-simplex $|B|$ in $C(m,2d+1)$ such that $A \wr B$, so that the two intersect transversely. But such a $B$ exists if and only if $A \in \chainset{m}{d}$.

Alternatively, one can apply Gale's Eveness Criterion, which entails that the vertices of an upper facet of $C(m,2d+1)$ consist of $m$ together with $d$ disjoint pairs of consecutive integers, and the vertices of a lower facet consist of $1$ together with $d$ disjoint pairs of consecutive integers. Hence $|A|$ is not contained in a facet if and only if $A \in \chainset{m}{d}$.
\end{proof}

The even-dimensional counterpart of the following result was shown in \cite[Proposition 2.13]{ot}.

\begin{lemma}\label{lem-simp-below}
Let $\mathcal{T} \in \mathcal{S}(m,2d+1)$. Let $|A|$ be an internal $d$-simplex of $\mathcal{T}$. Then there is a unique $(2d+1)$-simplex $|A\cup B|$ of $\mathcal{T}$ such that $B$ is a $(d+1)$-tuple with $B \wr A$.
\end{lemma}
\begin{proof}
Since $|A|$ is an internal $d$-simplex of $\mathcal{T}$, the points immediately below $|A|$ must lie in a unique $(2d+1)$-simplex $|S|$. Then $|A|$ is a $d$-face of $|S|$, and hence is the intersection of $d+1$ facets of $|S|$. The simplex $|S|$ has $d+1$ upper facets and $d+1$ lower facets, by Gale's Evenness Criterion. Then $|A|$ must be the intersection of the upper facets of $|S|$, otherwise $|A|$ lies in a lower facet of $|S|$, and so $|S|$ cannot contain the points immediately below $|A|$. But then, by Lemma \ref{up-low-fac}, we must have that $S = A \cup B$, where $B \wr A$. 
\end{proof}

By \cite{dey}, we know that it is possible to reconstruct a triangulation of a point configuration in $\mathbb{R}^{\delta}$ on the basis of knowing  only its $\lfloor \frac{\delta}{2} \rfloor$-faces. Hence we can reconstruct a triangulation of $C(m,2d+1)$ from its $d$-simplices alone. However, in the manner of \cite[Lemma 2.15]{ot}, we affirm this result by showing what the reconstructed triangulation looks like.

\begin{lemma}\label{lem-dey}
Let $\mathcal{T} \in \mathcal{S}(m,2d+1)$. Then $\mathcal{T}$ is determined by $\simp(\mathcal{T})$. In particular,
\begin{enumerate}
\item the $(d+1)$-simplices of $\mathcal{T}$ are those\label{d+1}
	\begin{itemize}
	\item whose $d$-faces $|A|$ are either such that $A \in \simp(\mathcal{T})$ or such that $A \notin \chainset{m}{d}$, and
	\item which are such that there is no $d$-simplex of $\mathcal{T}$ which intersects them transversely;
	\end{itemize} 
\item the $k$-simplices of $\mathcal{T}$ for $k>d+1$ are those whose $(d+1)$-faces satisfy (\ref{d+1}).
\end{enumerate}
\end{lemma}
\begin{proof}
It follows from Gale's Evenness Criterion that for $k<d$, every $k$-simplex lies on the boundary of $C(m,2d+1)$, and hence can be ignored.

Let $|A|$ be a $(d+1)$-simplex of $\mathcal{T}$. Then clearly $|A|$ cannot intersect any $d$-simplices of $\mathcal{T}$ transversely. Moreover, every $d$-face $|B|$ of $|A|$ is either internal, so that $B \in \simp(\mathcal{T})$, or not internal, so that $B \notin \chainset{m}{d}$. If $|A|$ is a $k$-simplex of $\mathcal{T}$ for $k>d+1$, then all of the $(d+1)$-faces of $|A|$ must satisfy (\ref{d+1}) for these reasons.

Conversely, if $|A|$ is not a $k$-simplex of $\mathcal{T}$ for some $k>d+1$, then $|A|$ must intersect a $(2d+1)$-simplex of $\mathcal{T}$ transversely. By the description of the circuits of $C(m,2d+1)$, either $|A|$ has a $d$-face $|A_{d}|$ which intersects a $(d+1)$-simplex $|B_{d+1}|$ of $\mathcal{T}$ transversely, or $|A|$ has a $(d+1)$-face $|A_{d+1}|$ which intersects a $d$-simplex $|B_{d}|$ of $\mathcal{T}$ transversely. In the first case, $A_{d}$ cannot be in $\simp(\mathcal{T})$, so any $(d+1)$-face of $|A|$ containing $|A_{d}|$ does not satisfy (\ref{d+1}). In the second case, $|A_{d+1}|$ does not satisfy (\ref{d+1}).
\end{proof}

It is useful to think of increasing bistellar flips in the following way in odd dimensions. This is the odd-dimensional version of Lemma \ref{up-low-fac}.

\begin{lemma}\label{lem-odd-flips}
Let $A \in \subs{m}{d+1}{}$ and $B \in \subs{m}{d+2}{}$. Then $A \wr B$ if and only if we have that $|A \cup B|_{2d+2}$ is a $(2d+2)$-simplex with $|A|_{2d+2}$ the intersection of its lower facets, and $|B|_{2d+2}$ the intersection of its upper facets.
\end{lemma}
\begin{proof}
Let $A \cup B =: S$. If $\# S < 2d+3$, then $A$ and $B$ cannot be intertwining. Hence suppose that $\# S =2d+3$, so that $|A \cup B|_{2d+2}$ is a $(2d+2)$-simplex. We then apply Gale's Evenness Criterion. The vertices of a lower facet of $C(S, 2d+2)$ miss out an even entry of $S$. The intersection of these tuples is $(s_{1},s_{3}, \dots, s_{2d+1})$. The vertices of an upper facet of $C(A \cup B, 2d+2)$ miss out an odd entry of $S$. Their intersection is $(s_{0}, s_{2}, \dots, s_{2d+2})$. Therefore $A \wr B$ if and only if $|A|_{2d+2}$ is the intersection of the lower facets of $|S|_{2d+2}$ and $|B|_{2d+2}$ is the intersection of its upper facets.
\end{proof}

\begin{corollary}\label{cor-odd-flips}
Let $A \in \subs{m}{d+1}{}$ and $B \in \subs{m}{d+2}{}$ such that $A\wr B$. Then the only internal $d$-simplex of the lower triangulation of $C(A\cup B,2d+1)$ is $|A|$, and the upper triangulation of $C(A \cup B, 2d+1)$ has no internal $d$-simplices.
\end{corollary}
\begin{proof}
This follows from Lemma \ref{lem-int-odd} and Lemma \ref{lem-odd-flips}.
\end{proof}

Hence we think of bistellar flips in odd dimensions as replacing a $d$-simplex with a $(d+1)$-simplex which intersects it transversely. Note that these form two halves of a circuit.

We now derive the properties which characterise triangulations of odd-dimensional cyclic polytopes. The following lemma is shown for $d=1$ in \cite[Lemma 4.3]{er}.

\begin{lemma}\label{support}
Let $\mathcal{T} \in \mathcal{S}(m,2d+1)$. Suppose that $|A|$ is an internal $d$-simplex of $\mathcal{T}$. Then there is a $d$-tuple $A'$ such that $A' \wr A$ and for every $(d+1)$-tuple $B \subset A \cup A'$ we have that $|B|$ is a $d$-simplex of $\mathcal{T}$.
\end{lemma}
\begin{proof}
By \cite[Theorem 1.1]{rambau}, $\mathcal{T}$ can be represented by a chain of triangulations $\mathcal{T}_{0}\lessdot_{1}\dots\lessdot_{1}\mathcal{T}_{r}$ in $\mathcal{S}_{1}(m,2d)$. There is therefore a triangulation $\mathcal{T}_{i}$ of $C(m,2d)$ which is a section of $\mathcal{T}$ and which contains the $d$-simplex $|A|$. By \cite[Proposition 2.13]{ot}, there is a $2d$-simplex $|A \cup A'|$ of $\mathcal{T}_{i}$, where $A' \wr A$. This $2d$-simplex $|A \cup A'|$ is, moreover, a $2d$-simplex of $\mathcal{T}$. It then follows that every $d$-face of $|A \cup A'|$ is a $d$-simplex of $\mathcal{T}$.
\end{proof}

\begin{definition}\label{def-support}
Let $X \subseteq \chainset{m}{d}$. We say that $X$ is \emph{supporting} if for any $A \in X$ there is a $d$-tuple $A'$ such that $A' \wr A$ and, for every $(d+1)$-tuple $B \subset A \cup A'$ such that $B \in \chainset{m}{d}$, we have that $B \in X$.
\end{definition}

The inspiration for the following lemma comes from \cite[Proposition 3.3, D3, and Proposition 4.2, T3]{er}, which concern simpler versions of the property for dimensions 2 and 3, applied to submersion sets.

\begin{lemma}\label{bridge}
Let $\mathcal{T} \in \mathcal{S}(m,2d+1)$. Let \[|A|:=|x_{0}, \dots, x_{i-1},a_{i}, \dots, a_{j}, x_{j+1}, \dots, x_{d}|,\,|B|:=|x_{0}, \dots, x_{i-1},b_{i}, \dots, b_{j}, x_{j+1}, \dots, x_{d}|\] be internal $d$-simplices of $\mathcal{T}$, where possibly $i=0$ or $j=d$, or both. Suppose these are such that $(a_{i}, \dots, a_{j}) \wr (b_{i}, \dots, b_{j})$. Then \[|S_{k}|:=|x_{0}, \dots, x_{i-1},a_{i}, \dots, a_{k-1}, b_{k}, \dots, b_{j},x_{j+1} \dots, x_{d}|\] is an internal $d$-simplex of $\mathcal{T}$ for all $i \leqslant k \leqslant j+1$.
\end{lemma}
\begin{proof}
Firstly, note that if $A, B \in \chainset{m}{d}$, then $S_{k} \in \chainset{m}{d}$.

We use induction on increasing bistellar flips of the triangulation: all triangulations of $C(m,2d+1)$ can be reached via increasing bistellar flips from the lower triangulation by \cite[Theorem 1.1]{rambau}. In the base case, which is the lower triangulation of $C(m,2d+1)$, all $d$-simplices are simplices of the triangulation $\mathcal{T}$. Hence the result holds trivially in this case.

We use contradiction to show the inductive step. Suppose that we perform an increasing bistellar flip on $\mathcal{T}$ by removing the $d$-simplex $|S_{k}|=|x_{0}, \dots, x_{i-1},a_{i}, \dots, a_{k-1}, b_{k}, \dots, b_{j},x_{j+1} \dots, x_{d}|$ for $k \in (i,j)$, so that we replace it with a $(d+1)$-simplex $|F|=|f_{0}, \dots, f_{d+1}|$ such that $S_{k} \wr F$. Then $|F|$ cannot intersect $|S_{k'}|$ transversely for $k' \neq k$, since, by the induction hypothesis, these are $d$-simplices of $\mathcal{T}$. Thus we must have $a_{k}\leqslant f_{k}$ and $f_{k} \leqslant b_{k-1}$, otherwise $S_{k+1} \wr F$ or $S_{k-1} \wr F$, respectively. But this is a contradiction, since $b_{k-1}<a_{k}$. Hence we can never perform an increasing bistellar flip by removing $S_{k}$, which means that the above property must be preserved by increasing bistellar flips.
\end{proof}

\begin{definition}\label{def-bridge}
Let $X \subseteq \chainset{m}{d}$. We say that $X$ is \emph{bridging} if whenever \[A:=(x_{0}, \dots, x_{i-1},a_{i}, \dots, a_{j}, x_{j+1}, \dots, x_{d}),\,B:=(x_{0}, \dots, x_{i-1},b_{i}, \dots, b_{j}, x_{j+1}, \dots, x_{d})\in X,\] where possibly $i=0$ or $j=d$, or both, such that $(a_{i}, \dots, a_{j}) \wr (b_{i}, \dots, b_{j})$, we have that \[S_{k}:=(x_{0}, \dots, x_{i-1},a_{i}, \dots, a_{k-1}, b_{k}, \dots, b_{j},x_{j+1} \dots, x_{d}) \in X\] for all $i \leqslant k \leqslant j+1$.
\end{definition}

Our strategy is now to see how the properties of being supporting and being bridging are affected by operations on triangulations. This will allow us to inductively construct a triangulation from a set of tuples which is supporting and bridging.

\begin{definition}
Let $X \subseteq \chainset{m}{d}$. We define
\begin{align*}
X/1 &:= \{ (a_{0}, \dots, a_{d}) \in X \mid a_{0} \neq 2\}, \\
X\backslash\{1,2\} &:= \{(a_{1}, \dots, a_{d}) \mid (2,a_{1}, \dots, a_{d}) \in X \}.
\end{align*}
Given a tuple $A \in X$ with $1 \notin A$, we define $1 \star A := \{1\}\cup A$. If $1 \notin A$ for all $A \in X$, we define $1\star X:=\{1\star A \mid A \in X\}$. We similarly define $2 \star X$, and so on.
\end{definition}

Note here that our notation is the same as \cite[Definition 2.17]{ot}, but our operations are different. This is because we want these operations on sets of tuples to correspond to the following operations on triangulations. These are the same as considered in \cite[Definition 2.16]{ot}, only we consider them in odd dimensions.

For a triangulation $\mathcal{T}$ of $C(m,\delta)$, we define $\mathcal{T}/1$ to be the triangulation of $C(m-1,\delta)$ obtained by continuously deforming $\mathcal{T}$ by moving vertex $1$ along the moment curve until it coincides with vertex $2$.

It is well-known that the triangulation of the vertex figure at 1 induced by a triangulation of $C(m,\delta)$ gives a triangulation of $C([2,m],\delta-1)$ \cite[Lemma 3.1]{rs-baues}. For a triangulation $\mathcal{T}$ of $C(m,\delta)$, we define $\mathcal{T}\backslash 1$ to be the triangulation of $C([2,m],\delta-1)$ given by the triangulation of the vertex figure at 1 of $\mathcal{T}$. Similarly, we define $\mathcal{T}\backslash\{1,2\}:=(\mathcal{T}\backslash 1)\backslash 2$.

For a triangulation $\mathcal{T}$ of $C(Y,\delta-1)$ and $x \notin Y$, we denote by $x \star \mathcal{T}$ the set of $\delta$-simplices \[\{|x \star S| \mid |S| \text{ is a }(\delta-1)\text{-simplex of }\mathcal{T}\}\] in $C(\{x\}\cup Y, \delta)$. This is in general only a partial triangulation.

\begin{lemma}\label{lem-contract-desc}
Let $\mathcal{T} \in \mathcal{S}(m,2d+1)$. Then $\simp(\mathcal{T}/1)=\simp(\mathcal{T})/1$.
\end{lemma}
\begin{proof}
Let $A$ be a $(d+1)$-tuple in $\simp(\mathcal{T}/1)$. Then we cannot have $a_{0}=2$, otherwise $|A|$ is a boundary $d$-simplex. Hence $a_{0}>2$. But then the pre-image of $|A|$ under the contraction $1\rightarrow 2$ must be $|A|$. Therefore $A \in \simp(\mathcal{T})/1$.

Conversely, let $A$ be a $(d+1)$-tuple in $\simp(\mathcal{T})/1$. Then $A \in \simp(\mathcal{T})$ and $a_{0}>2$. Hence $|A|$ is unaffected by the contraction $1\rightarrow 2$, and so $A \in \simp(\mathcal{T}/1)$.
\end{proof}

\begin{lemma}\label{lem-delete-desc}
Let $\mathcal{T} \in \mathcal{S}(m,2d+1)$. Then $\simp(\mathcal{T}\backslash\{1,2\})=\simp(\mathcal{T})\backslash\{1,2\}$.
\end{lemma}
\begin{proof}
Let $A$ be a $d$-tuple in $\simp(\mathcal{T}\backslash\{1,2\})$. Then there must be a $(d+1)$-simplex $|1\star 2\star A|$ of $\mathcal{T}$. We must have $a_{0}>3$, since $|A|$ is internal in $C([3,m],2d+1)$. Therefore $2 \star A \in \simp(\mathcal{T})$, and so $A \in \simp(\mathcal{T})\backslash\{1,2\}$.

Conversely, let $A \in \simp(\mathcal{T})\backslash\{1,2\}$, so that $2 \star A \in \simp(\mathcal{T})$. Therefore $\mathcal{T}$ contains every $d$-face of the $(d+1)$-simplex $|1\star 2 \star A|$, since all the other $d$-faces lie on the boundary of $C(m,2d+1)$. Moreover, since there cannot be a $d$-simplex $B$ of $\mathcal{T}$ such that $B \wr 1\star 2 \star A$, we must have that this $(d+1)$-simplex is in $\mathcal{T}$ by Lemma \ref{lem-dey}. Hence, $|A|$ is a $d$-simplex of $\mathcal{T}\backslash\{1,2\}$. Furthermore, $a_{0}>3$ because $2 \star A \in \simp(\mathcal{T})$. Hence $A \in \simp(\mathcal{T}\backslash\{1,2\})$.
\end{proof}

\begin{lemma}\label{lem-contract-works}
Suppose that $X \subseteq \chainset{m}{d}$ is supporting and bridging. Then $X/1$ is also supporting and bridging.
\end{lemma}
\begin{proof}
We first show that $X/1$ must be supporting. Let $A \in X/1$. Then $A \in X$. Since $X$ is supporting, there must be $B$ such that $B \wr A$ and every separated $(d+1)$-tuple contained in $A \cup B$ is in $X$. But these separated $(d+1)$-tuples will also be contained in $X/1$, which is therefore also supporting.

Now we show that $X/1$ must be bridging. Let \[A:=(x_{0}, \dots, x_{i-1},a_{i}, \dots, a_{j}, x_{j+1}, \dots, x_{d}),\,B:=(x_{0}, \dots, x_{i-1},b_{i}, \dots, b_{j}, x_{j+1}, \dots, x_{d})\] $\in X/1$, where possibly $i=0$ or $j=d$, or both. Suppose these are such that $(a_{i}, \dots, a_{j}) \wr (b_{i}, \dots, b_{j})$. Then $A,B \in X$. Since $X$ is bridging, we must have \[S_{k}:=(x_{0}, \dots, x_{i-1},a_{i}, \dots, a_{k-1}, b_{k}, \dots, b_{j},x_{j+1} \dots, x_{d})\in X\] for all $i \leqslant k \leqslant j+1$. But then $S_{k} \in X/1$ for all $i \leqslant k \leqslant j+1$ since $x_{0} \neq 2$ by assumption.
\end{proof}

\begin{lemma}\label{lem-delete-works}
Suppose that $X\subseteq\chainset{m}{d}$ is supporting and bridging. Then $X\backslash\{1,2\}$ is also supporting and bridging.
\end{lemma}
\begin{proof}
We first show that $X\backslash\{1,2\}$ must be supporting. Let $A \in X\backslash\{1,2\}$. Then $A':=2\star A \in X$. Thus, since $X$ is supporting, there is a $d$-tuple $B'$ such that $B' \wr A'$ and every separated sub-$(d+1)$-tuple of $A' \cup B'$ is in $X$. Then $B=(b'_{1}, \dots, b'_{d-1})$ is such that $B \wr A$. Let $C$ be a separated $d$-tuple contained in $A \cup B$. Then $2 \star C$ is contained in $A' \cup B'$, and so is in $X$, since $c_{0}\geqslant a_{0}>3$. This implies that $C \in X\backslash\{1,2\}$, as desired.

We now show that $X\backslash\{1,2\}$ must be bridging. Let \[A:=(x_{1}, \dots, x_{i-1},a_{i}, \dots, a_{j}, x_{j+1}, \dots, x_{d}),\,B:=(x_{1}, \dots, x_{i-1},b_{i}, \dots, b_{j}, x_{j+1}, \dots, x_{d})\] $\in X\backslash\{1,2\}$, where possibly $i=0$ or $j=d$, or both. Suppose that these are such that $(a_{i}, \dots, a_{j}) \wr (b_{i}, \dots, b_{j})$. Then $A':=2\star A,B':=2\star B \in X$. Since $X$ is bridging, we must have \[S'_{k}:=(2,x_{1}, \dots, x_{i-1},a_{i}, \dots, a_{k-1}, b_{k}, \dots, b_{j},x_{j+1} \dots, x_{d})\in X\] for all $i \leqslant k \leqslant j+1$. But then \[S_{k}:=(x_{1}, \dots, x_{i-1},a_{i}, \dots, a_{k-1}, b_{k}, \dots, b_{j},x_{j+1} \dots, x_{d}) \in X\backslash\{1,2\}\] for all $i \leqslant k \leqslant j+1$.
\end{proof}

The following technical proposition is key to proving our characterisation of triangulations of odd-dimensional cyclic polytopes. We shall use it in the subsequent proposition to construct a triangulation whose internal $d$-simplices are given by a particular supporting and bridging set of tuples. This is the most difficult step in the proof comprised by this section.

\begin{proposition}\label{prop-key-sup}
Let $X\subseteq \chainset{m}{d}$ be supporting and bridging for $C(m,2d+1)$. We suppose that $X$ is such that there are triangulations $\mathcal{U} \in \mathcal{S}([2,m],2d+1)$ and $\mathcal{W} \in \mathcal{S}([3,m],2d-1)$ such that $\simp(\mathcal{U})=X/1$ and $\simp(\mathcal{W})=X\backslash\{1,2\}$. Then $\mathcal{W}$ is a section of $\mathcal{U}\backslash 2$.
\end{proposition}
\begin{proof}
For this it suffices to show that any $d$-simplex of $\mathcal{W}$ is a $d$-simplex of $\mathcal{U}\backslash 2$. This is because $\mathcal{U}\backslash 2$ is a triangulation of a $2d$-dimensional cyclic polytope, and hence is determined by its $d$-simplices by \cite[Lemma 2.15]{ot}. For $k>d$, a $k$-simplex $|A|_{2d-1}$ is a $k$-simplex of $\mathcal{W}$ if and only if all its $d$-faces are $d$-simplices of $\mathcal{W}$, by Lemma \ref{lem-dey}. Moreover, $|A|_{2d}$ is a $d$-simplex of $\mathcal{U}\backslash 2$ if and only if all its $d$-faces are $d$-simplices of $\mathcal{U}\backslash 2$, by \cite[Lemma 2.15]{ot}. Hence if $\mathcal{U}\backslash 2$ contains all the $d$-simplices of $\mathcal{W}$, it must contain all the higher-dimensional simplices of $\mathcal{W}$ as well.

Note that $d$-simplices of $\mathcal{U}\backslash 2$ result from $(d+1)$-simplices of $\mathcal{U}$ with $2$ as a vertex. Hence one can show that every $d$-simplex of $\mathcal{W}$ is a $d$-simplex of $\mathcal{U}\backslash 2$ by showing that every $(d+1)$-simplex of $2 \ast \mathcal{W}$ is a $(d+1)$-simplex of $\mathcal{U}$. In turn, by Lemma \ref{lem-dey}, one can show this by showing that every $d$-simplex of $2 \ast \mathcal{W}$ is a $d$-simplex of $\mathcal{U}$ and that no $d$-simplex of $\mathcal{U}$ intersects a $(d+1)$-simplex of $2 \ast \mathcal{W}$ transversely.	

We first show that no $d$-simplex of $\mathcal{U}$ intersects a $(d+1)$-simplex of $2 \ast \mathcal{W}$ transversely. Suppose that $|A|$ is a $d$-simplex of $\mathcal{U}$ such that $A \wr B$, where $|B|$ is a $(d+1)$-simplex of $2 \ast \mathcal{W}$. We then have that $(b_{1}, \dots, b_{d}) \in \simp(\mathcal{W})$, since $2 \leqslant b_{0}<a_{0}<b_{1}$, so that $b_{1}>3$, and $b_{d} < b_{d+1} \leqslant m$. This means that $(2,b_{1}, \dots, b_{d}) \in X$. Since $A \in \simp(\mathcal{U}) \subseteq X$, we have that $(2,a_{1}, \dots, a_{d}) \in X$ by applying the bridging condition to $(2,b_{1}, \dots, b_{d})$ and $A$. This implies that $(a_{1}, \dots, a_{d}) \in \simp(\mathcal{W})$. But $(a_{1}, \dots, a_{d}) \wr (b_{1}, \dots, b_{d+1})$, which is a $d$-simplex of $\mathcal{W}$, a contradiction.

Hence, we now show that every $d$-simplex of $2 \ast \mathcal{W}$ is a $d$-simplex of $\mathcal{U}$. It is clear that if $|2 \star E|$ is a $d$-simplex of $2 \star \mathcal{W}$, then $|2 \star E|$ is a $d$-simplex of $\mathcal{U}$. Indeed, $|2 \star E|$ is on the boundary of $C([2,m],2d+1)$.

Therefore, let $|F|$ be a $d$-simplex of $\mathcal{W}$, and hence of $2 \star \mathcal{W}$. We must show that $|F|$ is a $d$-simplex of $\mathcal{U}$. If $|F|$ is not a $d$-simplex of $\mathcal{U}$, then there must be a $(d+1)$-simplex $|G|$ of $\mathcal{U}$ such that $F \wr G$. Hence $|G_{i}|:=|g_{0}, \dots, g_{i-1}, \hat{g}_{i}, g_{i+1}, \dots, g_{d+1}|$ is a $d$-simplex of $\mathcal{U}$ for all $i \in \{0,\dots, d+1\}$.

Suppose first that $g_{0} \neq 2$. Then we must have that $G_{d+1} \in \simp(\mathcal{U}) \subseteq X$, since this is a separated $(d+1)$-tuple with $g_{0}>2$ and $g_{d}<g_{d+1} \leqslant m$. We know that $(2,f_{0}, \dots, f_{d-1}) \in X$, since $|f_{0}, \dots, f_{d-1}|$ must be an internal $(d-1)$-simplex of $\mathcal{W}$. This is because $f_{d-1} < f_{d} \leqslant m$ and $f_{0} > g_{0} \geqslant 3$. Then, since $X$ is bridging, we must have that $(2,g_{1}, \dots, g_{d}) \in X$ by applying the bridging condition to $G_{d+1}$ and $(2,f_{0}, \dots, f_{d-1})$. But then $|g_{1}, \dots, g_{d}|$ is a $(d-1)$-simplex of $\mathcal{W}$ which intersects $|F|$ transversely, a contradiction.

If $g_{0}=2$, then consider the following. We know that $(2,f_{1}, \dots, f_{d}) \in X$, since $|f_{1}, \dots, f_{d}|$ is an internal $(d-1)$-simplex of $\mathcal{W}$. This is because $f_{1}>f_{0} \geqslant 3 > 2=g_{0}$ and $f_{d}<g_{d+1} \leqslant m$. Since $X$ is supporting, there must exist a $d$-tuple $H \wr (2,f_{1}, \dots, f_{d})$, such that every sub-$(d+1)$-tuple of $H \cup (2,f_{1}, \dots, f_{d})$ is a $(d+1)$-tuple of $X$. If $h_{0}>f_{0}$, then note that $2 \star H$ must be a $(d+1)$-tuple of $X$, so that $|H|$ is a $(d-1)$-simplex of $\mathcal{W}$ with $H \wr F$, a contradiction. If $h_{0} \leqslant f_{0}$, then note that $(h_{0},f_{1}, \dots, f_{d})$ must be a $(d+1)$-tuple of $X$. We then have that $|h_{0},f_{1}, \dots, f_{d}|$ is a $d$-simplex of $\mathcal{U}$ which intersects $|G|$ transversely, since $h_{0}>2$ and $g_{0}=2$. This is another contradiction.

Therefore every $d$-simplex of $2\star\mathcal{W}$ is a $d$-simplex of $\mathcal{U}$, which gives us that $\mathcal{W}$ is indeed a section of $\mathcal{U}\backslash 2$, as desired.
\end{proof}

We can now inductively construct triangulations from supporting and bridging collections.

\begin{proposition}\label{prop-triang-exists}
Let $X \subseteq \chainset{m}{d}$ be supporting and bridging. Then there is a triangulation $\mathcal{T} \in \mathcal{S}(m,2d+1)$ such that $X=\simp(\mathcal{T})$.
\end{proposition}
\begin{proof}
We show this by induction. The base cases consist of the case where $d=0$ and the case where $m=2d+2$. For $d=0$, triangulations of $C(m,2d+1)$ are given by subsets of vertices from $\{2,\dots, m-1\}$. Since the properties of being supporting or bridging are trivial for $d=0$, this case holds. For $m=2d+2$, $C(m,2d+1)$ is a simplex and so uniquely triangulates itself. In this case, $\chainset{m}{d}$ is empty, and so the unique triangulation is given by the empty set. Therefore the base cases hold.

For the inductive step, we consider triangulations of $C(m,2d+1)$ and suppose that the claim holds for $C(m',2d'+1)$ whenever $m'<m$ or $d'<d$. By Lemma \ref{lem-contract-works}, $X/1$ is both supporting and bridging. Hence, by the induction hypothesis, there is a triangulation $\mathcal{U} \in \mathcal{S}([2,m],2d+1)$ such that $\simp(\mathcal{U})=X/1$. By Lemma \ref{lem-delete-works}, $X\backslash\{1,2\}$ is also supporting and bridging. Hence, by the induction hypothesis, there is a triangulation $\mathcal{W} \in \mathcal{S}([3,m],2d-1)$ such that $\simp(\mathcal{W})=X\backslash\{1,2\}$.

By Proposition \ref{prop-key-sup}, we have that $\mathcal{W}$ is a section of $\mathcal{U}\backslash 2$. By \cite[Lemma 4.7]{rs-baues} as formulated in \cite[Lemma 2.3]{thomas} or \cite[Proposition 2.22]{ot}, we therefore have a triangulation $\mathcal{T} \in \mathcal{S}(m,2d+1)$ such that $\mathcal{T}/1=\mathcal{U}$ and $\mathcal{T}\backslash\{1,2\}=\mathcal{W}$. Hence $X/1=\simp(\mathcal{U})=\simp(\mathcal{T}/1)=\simp(\mathcal{T})/1$, by Lemma \ref{lem-contract-desc}, and $X\backslash\{1,2\}=\simp(\mathcal{W})=\simp(\mathcal{T}\backslash\{1,2\})=\simp(\mathcal{T})\backslash\{1,2\}$, by Lemma \ref{lem-delete-desc}. Hence $X=\simp(\mathcal{T})$, as desired.
\end{proof}

This at last establishes Theorem \ref{thm-class-odd-dim}, which states that triangulations of $C(m,2d+1)$ are in bijection with subcollections of $\chainset{m}{d}$ which are supporting and bridging via $\mathcal{T} \mapsto \simp(\mathcal{T})$. Lemma \ref{bridge} and Lemma \ref{support} give us that $\simp(\mathcal{T})$ is supporting and bridging for every triangulation $\mathcal{T}$. Lemma \ref{lem-dey} tells us that the assignment $\mathcal{T} \mapsto \simp(\mathcal{T})$ is injective. Finally Proposition \ref{prop-triang-exists} tells us that this is a surjection.

\begin{remark}
Theorem \ref{thm-class-odd-dim} generalises the bijection obtained in \cite{fr} between triangulations of three-dimensional cyclic polytopes and persistent graphs. Here the supporting and bridging properties correspond to the defining properties of persistent graphs, known in \cite{fr} as the \emph{bar property} and the \emph{X-property} respectively. The problem of characterising the $d$-skeleton of triangulations of $(2d + 1)$-dimensional cyclic polytopes was raised as an open problem in the conclusion of \cite{fr}. Theorem \ref{thm-class-odd-dim} solves this problem.
\end{remark}

\begin{example}\label{ex-odd-desc}
Consider the cyclic polytope $C(6,3)$. This has six triangulations, $\mathcal{T}_{l}, \mathcal{T}_{1}, \mathcal{T}_{2},$ $\mathcal{T}'_{1}, \mathcal{T}'_{2}, \mathcal{T}_{u}$, where $\simp(\mathcal{T}_{l})=\{24,25,35\},\,\simp(\mathcal{T}_{1})=\{24,25\},\, \simp(\mathcal{T}'_{1})=\{35,25\},\, \simp(\mathcal{T}_{2})=\{24\},$ $\, \simp(\mathcal{T}'_{2})=\{35\},\, \simp(\mathcal{T}_{u})=\emptyset$. The set $\{24,35\}$ is not obtained because it is not bridging, for which it would need to contain $25$. The set $\{25\}$ is not obtained because it is not supporting. The options for the supporting tuple are $3$ and $4$, which require $35$ and $24$ respectively.
\end{example}

\section{The higher Stasheff--Tamari orders in odd dimensions}\label{sect-odd}

We now give combinatorial and algebraic interpretations of the higher Stasheff--Tamari orders on triangulations of odd-dimensional cyclic polytopes. To obtain the algebraic interpretations, we first show how triangulations of odd-dimensional cyclic polytopes arise in the representation theory of $A_{n}^{d}$. This gives the other half of the picture from \cite{ot}, which shows how triangulations of even-dimensional cyclic polytopes arise in the representation theory of $A_{n}^{d}$.

Along the way, we show how our results can be used to give new realisations of the minimal embeddings of the second higher Stasheff--Tamari orders into Boolean lattices from \cite{thomas}. The advantage of our realisations is that they are direct, whereas the construction in \cite{thomas} is recursive.

We continue to work predominantly in the tilting framework but, again, we could equally work in the cluster-tilting framework. We end by showing how our results apply in the cluster-tilting framework in Section \ref{sect-max-green}, where odd-dimensional triangulations correspond to equivalence classes of maximal green sequences.

By a \emph{maximal chain} of tilting $A_{n}^{d}$-modules, we mean a sequence $(T_{1}, \dots, T_{r})$ of tilting $A_{n}^{d}$-modules such that  $T_{1}$ is the basic tilting module of projectives, $T_{r}$ is the basic tilting module of injectives, and, for $i \in [r-1]$, $T_{i+1}$ is a left mutation of $T_{i}$. We denote by $\mathcal{CT}(A_{n}^{d})$ the set of maximal chains of tilting $A_{n}^{d}$-modules. Given a maximal chain $C$ of tilting $A_{n}^{d}$-modules, we denote the set of indecomposable summands of modules of $C$ by $\Sigma(C)$.

By $\widetilde{\mathcal{CT}}(A_{n}^{d})$, we denote the set of equivalence classes of $\mathcal{CT}(A_{n}^{d})$ under the equivalence relation $\sim$, where, for $C_{1}, C_{2} \in \mathcal{CT}(A_{n}^{d})$, $C_{1} \sim C_{2}$ if and only if $\Sigma(C_{1})=\Sigma(C_{2})$. For $C \in \mathcal{CT}(A_{n}^{d})$, we denote its equivalence class in $\widetilde{\mathcal{CT}}(A_{n}^{d})$ by $[C]$.

We first observe the following lemma.

\begin{lemma}\label{lem-proj-inj}
Let $A \in \modset{n+2d}{d}$. Then the $A_{n}^{d}$-module $M_{A}$ is neither projective nor injective if and only if $|A|$ is an internal $d$-simplex in $C(n+2d,2d+1)$.
\end{lemma}
\begin{proof}
By \cite[Theorem 3.6]{ot}, $M_{A}$ is projective if and only if $a_{0}=1$; and $M_{A}$ is injective if and only if $a_{d}=n+2d$. The result then follows from Lemma \ref{lem-int-odd}.
\end{proof}

Triangulations of odd-dimensional cyclic polytopes correspond to equivalence classes of maximal chains of tilting modules, as follows.

\begin{theorem}\label{odd-dim-triangs}
There is a bijection between $\mathcal{S}(n+2d,2d+1)$ and $\widetilde{\mathcal{CT}}(A_{n}^{d})$. Moreover, if a triangulation $\mathcal{T} \in \mathcal{S}(n+2d,2d+1)$ corresponds to an equivalence class of maximal chains $[C] \in \widetilde{\mathcal{CT}}(A_{n}^{d})$, then
\begin{enumerate}
\item there is a bijection between mutations in $C$ and $(2d+1)$-simplices of $\mathcal{T}$; and\label{op-tilt-1}
\item there is a bijection between the internal $d$-simplices of $\mathcal{T}$ and elements of $\Sigma(C)$ which are neither projective nor injective.\label{op-tilt-2}
\end{enumerate}
\end{theorem}
\begin{proof}
This follows from \cite[Theorem 1.1(ii)]{rambau}, which states that triangulations of $C(n+2d,2d+1)$ are in bijection with maximal chains in $\mathcal{S}_{1}(n+2d,2d)$ under an equivalence relation of differing by a permutation of bistellar flip operations. By Theorem \ref{alg-cor-1st}, elements of $\mathcal{CT}(A_{n}^{d})$ correspond to maximal chains in $\mathcal{S}_{1}(n+2d,2d)$. Hence let $C=(T_{1}, \dots, T_{r}) \in \mathcal{CT}(A_{n}^{d})$ correspond to a maximal chain $\mathcal{C}=(\mathcal{T}_{1}, \dots, \mathcal{T}_{r})$ of $\mathcal{S}_{1}(n+2d,2d)$ which gives a triangulation $\mathcal{T}$ of $C(n+2d,2d+1)$.

We first establish the claims (\ref{op-tilt-1}) and (\ref{op-tilt-2}) for $\mathcal{T}$ and $C$. The claim (\ref{op-tilt-1}) is straightforward, because the $(2d+1)$-simplices of $\mathcal{T}$ correspond to increasing bistellar flips in $\mathcal{C}$ by \cite[Theorem 1.1(ii)]{rambau}. Then these correspond to mutations by Theorem \ref{alg-cor-1st}.

For claim (\ref{op-tilt-2}), let $M_{A} \in \Sigma(C)$ be neither projective nor injective, so that $M_{A}$ is an indecomposable summand of $T_{i}$ for some $i$. Then \cite[Theorem 1.1(ii)]{rambau} implies that $|A|$ is a $d$-simplex of $\mathcal{T}$. By Lemma \ref{lem-proj-inj}, $|A|$ is an internal $d$-simplex of $\mathcal{T}$. Conversely, if $|A|$ is an internal $d$-simplex of $\mathcal{T}$, then by Rambau's theorem, there is a triangulation $\mathcal{T}_{i}$ such that $|A|_{2d}$ is an internal $d$-simplex of $\mathcal{T}_{i}$. This implies that $M_{A}$ is an indecomposable summand of $T_{i}$, and by Lemma \ref{lem-proj-inj}, it is neither projective nor injective. This establishes claim (\ref{op-tilt-2}).

Now we must show that maximal chains of tilting $A_{n}^{d}$-modules are equivalent if and only if they give the same triangulation of $C(n+2d,2d+1)$. Let $C' \in \mathcal{CT}(A_{n}^{d})$ correspond to a triangulation $\mathcal{T}'$ of $C(n+2d,2d+1)$. Suppose that $C \sim C'$. By claim (\ref{op-tilt-2}), since $\Sigma(C)=\Sigma(C')$, we have that $\simp(\mathcal{T})=\simp(\mathcal{T}')$. Hence, by Lemma \ref{lem-dey} we have that $\mathcal{T}=\mathcal{T}'$ as required.

Conversely, it is clear that if $C$ and $C'$ correspond to the same triangulation, then we must have $C \sim C'$. This is because if $C$ and $C'$ correspond to the same triangulation, then they must have in common all indecomposable summands which are neither projective nor injective, by claim (\ref{op-tilt-2}). But, since all indecomposable projectives and injectives must also be summands of both $C$ and $C'$, we have that $\Sigma(C)=\Sigma(C')$. 
\end{proof}

\begin{remark}
Theorem \ref{thm-class-odd-dim} therefore also classifies sequences of mutations of basic tilting $A_{n}^{d}$-modules from the projectives to the injectives, up to equivalence.
\end{remark}

\subsection{First order}\label{odd-first}

In order to prove our combinatorial and algebraic characterisations of the first higher Stasheff--Tamari order we use the following fact about triangulations of polytopes.

\begin{lemma}\label{lem-triang-fact}
Let $\mathcal{T}$ be a triangulation of a $\delta$-dimensional polytope $P$, with $\Sigma$ an internal $k$-simplex of $\mathcal{T}$. Then $\Sigma$ is the intersection of at least $\delta-k+1$ different $\delta$-simplices of $\mathcal{T}$.
\end{lemma}
\begin{proof}
We prove the result by downwards induction on $k$. Our base case is $k=\delta-1$. Here $(\delta-1)$-simplices are facets of $\delta$-simplices. A facet of a given $\delta$-simplex must either be a shared facet with another $\delta$-simplex, or lie within a boundary facet of $P$. Hence an internal $(\delta-1)$-simplex must be the intersection of at least two $\delta$-simplices.

For the inductive step, we assume that the result holds for $k+1$. Let $\Sigma$ be a $k$-simplex of $\mathcal{T}$ for $k<\delta-1$. Then $\Sigma$ is a face of a $\delta$-simplex, and so must be the intersection of an least two $(k+1)$-simplices. Moreover, $\Sigma$ cannot lie in any boundary $(k+1)$-simplices, otherwise it is a boundary $k$-simplex. Thus, let $\Sigma$ be the intersection of two internal $(k+1)$-simplices $\Sigma^{1}$ and $\Sigma^{2}$. By the induction hypothesis, both $\Sigma^{1}$ and $\Sigma^{2}$ are the intersection of $\delta$-simplices $\{\Sigma_{1}^{1}, \dots, \Sigma_{l_{1}}^{1}\}$ and $\{\Sigma_{1}^{2}, \dots, \Sigma_{l_{2}}^{2}\}$ respectively, where $l_{1},l_{2} \geqslant \delta-k$. Since $\Sigma^{1}$ and $\Sigma^{2}$ are distinct, we must have that \[\{\Sigma_{1}^{1}, \dots, \Sigma_{l_{1}}^{1}\} \neq \{\Sigma_{1}^{2}, \dots, \Sigma_{l_{2}}^{2}\}.\] Thus $\#\{\Sigma_{1}^{1}, \dots, \Sigma_{l_{1}}^{1},\Sigma_{1}^{2}, \dots, \Sigma_{l_{2}}^{2}\} \geqslant \delta-k+1$, and so $\Sigma$ is the intersection of at least $\delta-k+1$ different $\delta$-simplices of $\mathcal{T}$.
\end{proof}

We now give a combinatorial characterisation of the first higher Stasheff--Tamari order in terms of sets of $(d+1)$-tuples.

\begin{theorem}\label{thm-hst1-odd}
Let $\mathcal{T}, \mathcal{T}' \in \mathcal{S}(m,2d+1)$. Then we have $\mathcal{T} \lessdot_{1} \mathcal{T}'$ if and only if $\simp(\mathcal{T})=\simp(\mathcal{T}')\cup\{A\}$ for some $A \in \chainset{m}{d}\setminus \simp(\mathcal{T}')$.
\end{theorem}
\begin{proof}
Suppose first that $\mathcal{T}'$ is an increasing bistellar flip of $\mathcal{T}$. Then $\mathcal{T}$ and $\mathcal{T}'$ coincide everywhere but inside a copy of $C(2d+3,2d+1)$. This direction then follows from applying Corollary \ref{cor-odd-flips} to this copy of $C(2d+3,2d+1)$.

We now suppose that $\simp(\mathcal{T})=\simp(\mathcal{T}')\cup\{A\}$ for some $A \in \chainset{m}{d}\setminus \simp(\mathcal{T}')$. Since $|A|$ is not a $d$-simplex of $\mathcal{T}'$, there must be a $(d+1)$-simplex $|B|$ of $\mathcal{T}'$ such that $A \wr B$. Suppose that $|S|$ is a $(2d+1)$-simplex of $\mathcal{T}$ which has $|A|$ as a $d$-face. Suppose further that $|S|$ possesses a vertex $x \notin A \cup B$. If $x \in (b_{i-1},b_{i})$ for some $i \in [d+1]$, then $(A\setminus\{a_{i-1}\})\cup\{x\}=:A' \wr B$, which is a contradiction, since $A' \in \simp(\mathcal{T})\setminus \{A\}=\simp(\mathcal{T}')$. Similarly, if $x>b_{d+1}$, then $(b_{1}, \dots, b_{d+1}) \wr A \cup \{x\}$, which contradicts the fact that $(b_{1}, \dots, b_{d+1})\in \simp(\mathcal{T}') \subset \simp(\mathcal{T})$. The case $x<b_{0}$ can be treated in the same way.

Thus every $(2d+1)$-simplex $|S|$ of $\mathcal{T}$ with $|A|$ as a $d$-face has vertices in $A \cup B$. There are $d+2$ such $(2d+1)$-simplices, given by $S_{i}:=(A \cup B)\setminus\{b_{i}\}$ for each $b_{i} \in B$. The triangulation $\mathcal{T}$ must contain all of these $|S_{i}|$, since $A$ must be the intersection of at least $d+2$ different $(2d+1)$-simplices, by Lemma \ref{lem-triang-fact}. The set $\{|S_{i}|\}_{i=0}^{d+1}$ gives the lower triangulation of $C(A \cup B, 2d+1)$.  None of these $(2d+1)$-simplices can be contained in $\mathcal{T}'$, but every other $(2d+1)$-simplex of $\mathcal{T}$ must be contained in $\mathcal{T}'$ by Lemma \ref{lem-dey}. It then follows that $\mathcal{T}'$ must be obtained by replacing the lower triangulation of $C(A \cup B, 2d+1)$ with the upper triangulation, since these are the only two possible triangulations of $C(A \cup B, 2d+1)$. Hence $\mathcal{T}'$ is an increasing bistellar flip of $\mathcal{T}$.
\end{proof}

We now give an algebraic characterisation of the first higher Stasheff--Tamari order in terms of maximal chains of tilting $A_{n}^{d}$-modules. Our terminology is based on \cite{hi-no-gap}.

An \emph{oriented polygon} is a sub-poset of $\mathcal{S}_{1}(m,2d)$ formed of a union of chain of length $d+2$ with a chain of length $d+1$, such that these chains intersect only at the top and bottom. (For an illustration see Figure \ref{fig-flip}.) Here the \emph{length} of a chain is the number of covering relations in it. We think of an oriented polygon as being oriented from the longer side to the shorter side. If two maximal chains $C,C'$ differ only in that $C$ contains the longer side of an oriented polygon and $C'$ contains the shorter side, then we say that $C'$ is \emph{an increasing elementary polygonal deformation} of $C$. Note that an increasing elementary polygonal deformation decreases the length of the chain.

\begin{theorem}\label{thm-alg-odd-hst1}
Let $\mathcal{T}, \mathcal{T}' \in \mathcal{S}(n+2d,2d+1)$ correspond to equivalence classes of maximal chains of tilting modules $[C], [C'] \in \widetilde{\mathcal{CT}}(A_{n}^{d})$. Then $\mathcal{T}\lessdot_{1}\mathcal{T}'$ if and only if there are equivalence class representatives $\widehat{C} \in [C]$ and $\widehat{C'} \in [C']$ such that $\widehat{C'}$ is an increasing elementary polygonal deformation of $\widehat{C}$.
\end{theorem}

In the proof of this theorem we shall require the following tools. Following \cite[Definition 5.7]{rambau}, for $\delta$-simplices $\Sigma_{1}, \Sigma_{2}$ of a triangulation $\mathcal{T} \in \mathcal{S}(m,\delta)$ with $\delta$ vertices in common, we write that $\Sigma_{1} \imunder \Sigma_{2}$ if and only if $\Sigma_{1} \cap \Sigma_{2}$ lies in the upper facets of $\Sigma_{1}$ and the lower facets of $\Sigma_{2}$. The relation $\prec$ is defined as the transitive closure of $\imunder$, so that $\Sigma_{1} \dashrightarrow \Sigma_{2}$ implies that $\Sigma_{1} \prec \Sigma_{2}$. This is a partial order by \cite[Corollary 5.9]{rambau}.

\begin{proof}
Let $\mathcal{T}, \mathcal{T}'$ be triangulations of $C(n+2d,2d+1)$ corresponding respectively to $[C], [C'] \in \widetilde{\mathcal{CT}}(A_{n}^{d})$.

Suppose that $\mathcal{T}'$ is an increasing bistellar flip of $\mathcal{T}$. Let $S \in \binom{[n+2d]}{2d+3}$ be the $(2d+3)$-tuple of vertices giving the bistellar flip. Let $S_{i}:=(s_{0}, \dots, s_{i-1}, \hat{s}_{i}, s_{i+1}, \dots, s_{2d+2})$. The lower triangulation of $C(S,2d+1)$ consists of the $(2d+1)$-simplices $|S_{i}|$ for $i$ even and the upper triangulation consists of the $(2d+1)$-simplices $|S_{i}|$ for $i$ odd. Then $|S_{2j}| \imunder |S_{2i}|$ for $i < j$: $|S_{2i}| \cap |S_{2j}|$ is an upper facet of $|S_{2j}|$ and a lower facet of $|S_{2i}|$. Thus one can extend $\prec$ to the total order $\prec_{t}$ on the simplices of the lower triangulation of $C(S,2d+1)$ by \[|S_{2d+2}| \prec_{t} |S_{2d}| \prec_{t} \dots \prec_{t} |S_{0}|.\] This can be consistently extended to a total order on the $(2d+1)$-simplices of $\mathcal{T}$ which contains this chain as an interval. This would only be impossible if there were a $(2d+1)$-simplex $\Sigma$ of $\mathcal{T}$ such that $|S_{2j}| \prec \Sigma \prec |S_{2i}|$, where $i<j$. But, since  $|S_{2j}| \prec |S_{2i}|$ is a covering relation for $\prec$, we would have to have $\Sigma=|S_{2j}|$ or $\Sigma=|S_{2i}|$.

Therefore, by \cite[Corollary 5.12]{rambau}, there is a maximal chain $\widehat{\mathcal{C}}$ of $\mathcal{S}_{1}(n+2d,2d)$ corresponding to $\mathcal{T}$ such that the sequence of bistellar flips in $\widehat{\mathcal{C}}$ is \[(\Sigma_{1}, \dots, \Sigma_{r-1}, |S_{2d+2}|, |S_{2d}|, \dots, |S_{0}|, \Sigma'_{1}, \dots, \Sigma'_{s-1}).\] A similar argument shows that there exists a maximal chain $\widehat{\mathcal{C}'}$ of $\mathcal{S}_{1}(n+2d,2d)$ corresponding to $\mathcal{T}'$ such that the sequence of bistellar flips is \[(\Sigma_{1}, \dots, \Sigma_{r-1}, |S_{1}|, |S_{3}|, \dots, |S_{2d+1}|, \Sigma'_{1}, \dots, \Sigma'_{s-1}).\] Since the $(2d+1)$-simplices of $\mathcal{T}'$ outside $C(S,2d+1)$ are the same as those of $\mathcal{T}$, namely $\{\Sigma_{1}, \dots, \Sigma_{r-1}, \Sigma'_{1}, \dots, \Sigma'_{s-1}\}$, we may choose the same order on them in both maximal chains $\widehat{\mathcal{C}}$ and $\widehat{\mathcal{C}'}$. It follows from the description of triangulations of $C(2d+3,2d)$, see, for instance, \cite[Proof of Proposition 9.1]{thomas-bst}, that the chains in $\mathcal{S}_{1}(n+2d,2d)$ given here by $(|S_{2d+2}|, |S_{2d}|, \dots, |S_{0}|)$ and $(|S_{1}|, |S_{3}|, \dots, |S_{2d+1}|)$ intersect only at their top and bottom. Hence these chains form an oriented polygon.

Then, by Theorem \ref{odd-dim-triangs}, these correspond to $\widehat{C} \in \mathcal{CT}(A_{n}^{d})$, where \[\widehat{C}=(U_{1}, \dots, U_{r}, T_{1}, \dots, T_{d+1}, V_{1}, \dots, V_{s}),\] and $\widehat{C'} \in \mathcal{CT}(A_{n}^{d})$, where \[\widehat{C'}=(U_{1}, \dots, U_{r}, T'_{1}, \dots, T'_{d}, V_{1}, \dots, V_{s}).\] Thus $\widehat{C'}$ is an increasing elementary polygonal deformation of $\widehat{C}$, as required. Note that the $(2d+1)$-simplices in the sequences of bistellar flips of $\widehat{\mathcal{C}}$ and $\widehat{\mathcal{C}'}$ come in between the respective tilting modules of $\widehat{C}$ and $\widehat{C'}$, which correspond to triangulations.

Conversely, suppose that we have equivalence class representatives \[\widehat{C}=(U_{1}, \dots, U_{r}, T_{1}, \dots,T_{d+1}, V_{1}, \dots, V_{s})\] and \[\widehat{C'}=(U_{1}, \dots, U_{r}, T'_{1}, \dots, T'_{d}, V_{1}, \dots, V_{s})\] in $\mathcal{CT}(A_{n}^{d})$. Here, as before, let $\widehat{C}$ give the triangulation $\mathcal{T}$ and $\widehat{C'}$ give the triangulation $\mathcal{T}'$. We claim that this implies that $\mathcal{T}'$ is an increasing bistellar flip of $\mathcal{T}$. By Theorem \ref{odd-dim-triangs}, by transforming $\widehat{C}$ into $\widehat{C'}$, we have removed $d+2$ different $(2d+1)$-simplices $\{\Sigma_{0},\Sigma_{2} \dots, \Sigma_{2d+2}\}$ from $\mathcal{T}$ and replaced them by $d+1$ different $(2d+1)$-simplices $\{\Sigma_{1}, \Sigma_{3}, \dots, \Sigma_{2d+1}\}$. We can suppose that $\Sigma_{1}$ is not in the triangulation $\mathcal{T}$, since at least one of these simplices must not be. Hence it must intersect a $(2d+1)$-simplex of the triangulation $\mathcal{T}$ transversely, and so it must intersect $\Sigma_{2l}$ for some $l$.

Hence, there is a circuit $(A,B)$ with $|A| \subset \Sigma_{2l}$, $|B| \subset \Sigma_{1}$. By the description of the circuits of a cyclic polytope, one of $|A|,|B|$ is an internal $d$-simplex and the other is an internal $(d+1)$-simplex. Internal $d$-simplices are intersections of at least $d+2$ different $(2d+1)$-simplices and internal $(d+1)$-simplices are intersections of at least $d+1$ different $(2d+1)$-simplices, by Lemma \ref{lem-triang-fact}. Therefore, to remove an internal $d$-simplex, one must remove the $d+2$ different $(2d+1)$-simplices whose intersection it is. Hence, since $|A|$ and $|B|$ intersect each other transversely and so cannot be in the same triangulation, we must have $|A| = \bigcap_{i=0}^{d+1}\Sigma_{2i}$ and $|B| = \bigcap_{j=0}^{d}\Sigma_{2j+1}$. Moreover, the only internal $d$-simplex we can have removed from $\mathcal{T}$ is the intersection of these $d+2$ different $(2d+1)$-simplices, which is $A$, and so the remaining internal $d$-simplices of $\mathcal{T}$ are internal $d$-simplices of $\mathcal{T}'$. Thus $\mathcal{T}'$ is an increasing bistellar flip of $\mathcal{T}$ by Theorem \ref{thm-hst1-odd}. 
\end{proof}

\begin{figure}
\caption{An increasing elementary polygonal deformation of maximal chains of tilting modules.}\label{fig-flip}
\[\begin{tikzcd}
&& \vdots && \\
&& V_{1} \ar[u,rightsquigarrow] && \\
& T_{d+1} \ar[ur,rightsquigarrow] && T'_{d} \ar[ul,rightsquigarrow] & \\
& \vdots \ar[u,rightsquigarrow] \ar[Rightarrow, rr] && \vdots \ar[u,rightsquigarrow] && \\
& T_{1} \ar[u,rightsquigarrow] && T'_{1} \ar[u,rightsquigarrow] & \\
&& U_{r} \ar[ul,rightsquigarrow] \ar[ur,rightsquigarrow] && \\
&& \vdots \ar[u,rightsquigarrow] &&
\end{tikzcd}\]
\end{figure}
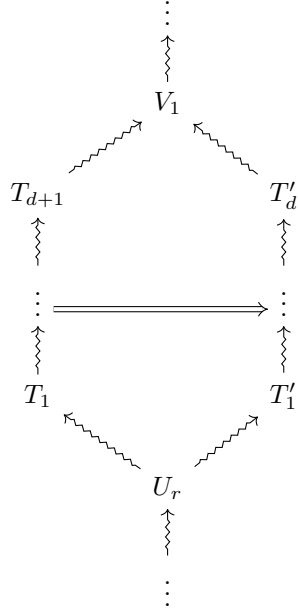

\begin{remark}
An \emph{$n$-category} is a category enriched in $(n-1)$-categories, where an ordinary category is a 1-category. By Theorem \cite[Theorem 4.4]{ot}, we have that tilting $A_{n}^{d}$-modules correspond bijectively to triangulations of $C(n+2d,2d)$. By Theorem \cite[Theorem 3.4]{kv-poly}, triangulations of $C(n+2d,2d)$ form an $(n-1)$-category. Hence the set of tilting $A_{n}^{d}$-modules forms an $(n-1)$-category. Indeed, the irreducible 1-morphisms of this category are left mutations, and the irreducible 2-morphisms are the increasing elementary polygonal deformations of equivalence classes of maximal chains from Theorem \ref{thm-alg-odd-hst1}.
\end{remark}

\subsection{Second order}\label{odd-second}

To obtain our combinatorial interpretation of the second higher Stasheff--Tamari order we first give an alternative to the interpretation of Edelman and Reiner in terms of submersion sets.

\begin{definition}
Let $\Sigma$ be a $k$-simplex in $C(m,\delta)$. Given a triangulation $\mathcal{T} \in \mathcal{S}(m, \delta)$, we say that $\mathcal{T}$ \emph{supermerges} $\Sigma$ if for all $x \in \Sigma$, \[s_{\Sigma}(x)_{\delta+1} \geqslant s_{\mathcal{T}}(x)_{\delta+1}.\] We then define the \emph{$k$-supermersion set of $\mathcal{T}$} to be \[\mathrm{sup}_{k}\mathcal{T}:= \left\lbrace |A| \mathrel{\Big|} A \in \subs{m}{k+1}{},~ \mathcal{T} \text{ supermerges } |A| \right\rbrace.\]
\end{definition}

\begin{remark}\label{rmk-d-simp-low-fac}
Consider $\mathrm{sup}_{d}\mathcal{T}$ for a triangulation $\mathcal{T} \in \mathcal{S}(m,2d+1)$. Every $d$-simplex in $C(m,2d+2)$ lies in a lower facet by Gale's Evenness Criterion. This is because the vertex tuple of a lower facet of $C(m,2d+2)$ is a disjoint union of $d+1$ pairs of consecutive numbers from $m$; any subset of $[m]$ of size $d+1$ is therefore a subset of a vertex tuple of a lower facet. Therefore no points in a $d$-simplex $|A|_{2d+2}$ can lie strictly above a section given by a triangulation. Hence, for any triangulation $\mathcal{T} \in \mathcal{S}(m, 2d+1)$, the supermersion set $\mathrm{sup}_{d}\mathcal{T}$ is precisely the set of $d$-simplices of $\mathcal{T}$.
\end{remark}

\begin{lemma}\label{lem-supset-d+1}
Let $\mathcal{T} \in \mathcal{S}(m,2d+1)$. Then, given a $(d+1)$-simplex $|A|$, we have that $|A| \in \mathrm{sup}_{d+1}\mathcal{T}$ if and only if every $d$-face of $|A|$ is in $\mathrm{sup}_{d}\mathcal{T}$. 
\end{lemma}
\begin{proof}
The forwards direction is clear. Conversely, suppose that every $d$-face of $|A|$ is in $\mathrm{sup}_{d}\mathcal{T}$. Suppose for contradiction that $|A| \notin \mathrm{sup}_{d+1}\mathcal{T}$. Then either $s_{|A|}(x)_{2d+2} < s_{\mathcal{T}}(x)_{2d+2}$ for all $x \in \mathring{|A|}$, or $\mathring{|A|}$ intersects $s_{\mathcal{T}}(C(m,2d+1))$.

In the first case, $|A|$ cannot be a $(d+1)$-simplex of $\mathcal{T}$. Hence either there is a $d$-simplex of $\mathcal{T}$ which intersects $|A|$ transversely, or there is a $d$-face of $|A|$ which intersects a $(d+1)$-simplex of $\mathcal{T}$ transversely. But we cannot be in the second case, since every $d$-face of $|A|$ is a $d$-simplex of $\mathcal{T}$ by Remark \ref{rmk-d-simp-low-fac}. Therefore there is a $d$-simplex $|B|$ of $\mathcal{T}$ such that $B \wr A$, and so $|A|$ and $|B|$ intersect in a unique point $y$. But then $|B|$ is the intersection of the lower facets of $|A \cup B|_{2d+2}$ and $|A|$ is the intersection of the upper facets, by Lemma \ref{lem-odd-flips}. Therefore $s_{|A|}(y)_{2d+2} > s_{\mathcal{T}}(y)_{2d+2}$, which contradicts our assumption in this case.

In the second case, by the description of the circuits of $C(m,2d+2)$, there must be a $(d+1)$-simplex $|B|$ of $\mathcal{T}$ such that either $A \wr B$ or $B \wr A$. For the first option here let $(a_{1}, \dots, a_{d+1})$ and for the second option here let $A'=(a_{0}, \dots, a_{d})$. But then we have that $|A'|$ is a $d$-simplex of $\mathcal{T}$ and $A' \wr B$, which is a contradiction.
\end{proof}

We now prove the following theorem, which gives us an interpretation of the second higher Stasheff--Tamari order on $(2d+1)$-dimensional cyclic polytopes in terms of $d$-simplices.

\begin{theorem}\label{thm-supermersion}
Let $\mathcal{T}, \mathcal{T}' \in \mathcal{S}(m,\delta)$. Then $\mathcal{T} \leqslant_{2} \mathcal{T}'$ if and only if \[\mathrm{sup}_{\lfloor \frac{\delta}{2}\rfloor} \mathcal{T} \supseteq \mathrm{sup}_{\lfloor \frac{\delta}{2} \rfloor}\mathcal{T}'.\]
\end{theorem}
\begin{proof}
We know from \cite{er} that $\mathcal{T} \leqslant_{2} \mathcal{T}'$ if and only if \[\mathrm{sub}_{\lceil \frac{\delta}{2}\rceil} \mathcal{T} \subseteq \mathrm{sub}_{\lceil \frac{\delta}{2} \rceil}\mathcal{T}'.\] In the case where $\delta$ is even, \[\bigfloor{ \frac{\delta}{2}} = \bigceil{\frac{\delta}{2}},\] so the result simply follows from the symmetry that exists in the even case via the permutation \[\alpha:=\begin{pmatrix}
1 & 2 & \dots & m \\
m & (m-1) & \dots & 1
\end{pmatrix}.\] By \cite[Proposition 2.11]{er} this gives an order-reversing bijection on $\mathcal{S}_{2}(m,2d)$. We write $\alpha\mathcal{T}$ and $\alpha\mathcal{T}'$ for the images of the respective triangulations under the permutation $\alpha$. By Proposition \ref{desc-sub} and its dual, $\mathrm{sub}_{d}\alpha\mathcal{T}=\alpha\mathrm{sup}_{d}\mathcal{T}$. Hence
\begin{align*}
\mathrm{sup}_{d}\mathcal{T} \supseteq \mathrm{sup}_{d}\mathcal{T}' &\iff \alpha\mathrm{sub}_{d}\alpha\mathcal{T} \supseteq \alpha\mathrm{sub}_{d}\alpha\mathcal{T}' \\
&\iff \mathrm{sub}_{d}\alpha\mathcal{T} \supseteq \mathrm{sub}_{d}\alpha\mathcal{T}' \\
&\iff \alpha\mathcal{T} \geqslant_{2} \alpha\mathcal{T}' \\
&\iff \mathcal{T} \leqslant_{2} \mathcal{T}'.
\end{align*} 

In the case where $\delta$ is odd, we also proceed using the interpretation of the second higher Stasheff--Tamari order in terms of submersion sets. Suppose that $\mathrm{sup}_{d} \mathcal{T} \supseteq \mathrm{sup}_{d}\mathcal{T}'$. Let $|A| \in  \mathrm{sub}_{d+1}\mathcal{T}$. There are three cases:
\begin{enumerate}
\item $|A| \in \mathrm{sub}_{d+1}\mathcal{T}'$,\label{1first}
\item $|A|_{2d+2}$ intersects $s_{\mathcal{T}'}(C(m,2d+1))$, and\label{1second}
\item $|A| \in \mathrm{sup}_{d+1}\mathcal{T}'$.\label{1third}
\end{enumerate}
We want to establish that we are in case (\ref{1first}). Hence, suppose that we are in case (\ref{1second}). Then, by the description of the circuits of $C(m,2d+2)$, there must be a $(d+1)$-simplex $B$ of $\mathcal{T}'$ such that either $B \wr A$, in which case let $B'=(b_{1}, \dots, b_{d+1})$, or $A \wr B$, in which case let $B'=(b_{0}, \dots, b_{d})$. Then, in either case, $|B'|$ is a $d$-simplex of $\mathcal{T}'$, so $|B'| \in \mathrm{sup}_{d}\mathcal{T}' \subseteq \mathrm{sup}_{d}\mathcal{T}$. Therefore $|B'|$ is a $d$-simplex of $\mathcal{T}$. But $B' \wr A$, so that $|B'|_{2d+2}$ is the intersection of the lower facets of the $(2d+2)$-simplex $|A \cup B'|_{2d+2}$, and $|A|$ is the intersection of its upper facets, by Lemma \ref{lem-odd-flips}. But this contradicts the fact that $|A|$ is submerged by $\mathcal{T}$.

Now we suppose that we are in case (\ref{1third}). Therefore if $|A'|$ is a $d$-face of $|A|$, then $|A'| \in \mathrm{sup}_{d}\mathcal{T}' \subseteq \mathrm{sup}_{d}\mathcal{T}$. Thus $|A| \in \mathrm{sup}_{d+1}\mathcal{T}$, by Lemma \ref{lem-supset-d+1}, since all its $d$-faces are in $\mathrm{sup}_{d}\mathcal{T}$. This means that $|A|$ is a $(d+1)$-simplex of $\mathcal{T}$, as $|A| \in (\mathrm{sup}_{d+1}\mathcal{T}) \cap (\mathrm{sub}_{d+1}\mathcal{T})$.

We may suppose that $|A|$ is not a $(d+1)$-simplex of $\mathcal{T}'$, since otherwise $|A| \in \mathrm{sub}_{d+1}\mathcal{T}'$ automatically. Hence, there must be a $d$-simplex $|B|$ of $\mathcal{T}'$ such that $B \wr A$, because every $d$-face of $|A|$ is a $d$-simplex of $\mathcal{T}'$, so a $(d+1)$-simplex of $\mathcal{T}'$ cannot intersect a $d$-face of $|A|$ transversely. Then $|B| \in \mathrm{sup}_{d}\mathcal{T}' \subseteq \mathrm{sup}_{d}\mathcal{T}$. But this is a contradiction because $|B|$ cannot be a $d$-simplex of $\mathcal{T}$, since it intersects $|A|$ transversely. Thus $\mathrm{sub}_{d+1}\mathcal{T} \subseteq \mathrm{sub}_{d+1}\mathcal{T}'$, as desired.

Now we suppose that $\mathcal{T} \leqslant_{2} \mathcal{T}'$. Let $|A| \in \mathrm{sup}_{d}\mathcal{T}'$. Then $|A|_{2d+2}$ cannot intersect $s_{\mathcal{T}}(C(m,2d+1))$ transversely, since it is too small: a circuit in $C(m,2d+2)$ consists of a pair of $(d+1)$-simplices. Therefore, we suppose for contradiction that $|A| \in \mathrm{sub}_{d}\mathcal{T}\setminus\mathrm{sup}_{d}\mathcal{T}$. This means that for all $x \in \mathring{|A|}$, \[s_{|A|}(x)_{2d+2} < s_{\mathcal{T}}(x)_{2d+2} \leqslant s_{\mathcal{T}'}(x)_{2d+2} = s_{|A|}(x)_{2d+2},\] which is a contradiction. Hence $\mathrm{sup}_{d}\mathcal{T}' \subseteq \mathrm{sup}_{d}\mathcal{T}$.
\end{proof}

\begin{remark}
One could, of course, consider complements of supermersion sets instead of supermersion sets. Since $d$-simplices in $C(n,2d+2)$ all lie on the lower facets, these would comprise the $d$-simplices which are \emph{strictly submerged} by a triangulation $\mathcal{T}$, that is: submerged by $\mathcal{T}$ without being a $d$-simplex of $\mathcal{T}$. The inclusion of these sets would be in the same direction as the second higher Stasheff--Tamari order, so some may have an aesthetic preference for this approach.

However, $d$-supermersion sets of triangulations $\mathcal{T}$ of $C(m,2+1)$ are more natural objects to consider. As written above, these are simply the $d$-simplices of $\mathcal{T}$. They also fit more naturally into our algebraic description of the higher Stasheff--Tamari orders in odd dimensions in Theorem \ref{thm-alg-odd-hst2}.
\end{remark}

\begin{remark}
Edelman and Reiner describe $\mathcal{S}_{2}(m,\delta)$ in terms of $\ceil{\frac{\delta}{2}}$-submersion sets. But \cite{dey} tells us that a triangulation of $C(m,\delta)$ is determined by its $\floor{\frac{\delta}{2}}$-simplices. Logically, then, the second higher Stasheff--Tamari order ought also to be controlled by $\floor{\frac{\delta}{2}}$-simplices. Theorem \ref{thm-supermersion} shows us that this is indeed the case.
\end{remark}

\begin{corollary}\label{thm-hst2-odd}
Let $\mathcal{T},\mathcal{T}' \in \mathcal{S}(m,2d+1)$. Then $\mathcal{T} \leqslant_{2} \mathcal{T}'$ if and only if $\simp(\mathcal{T}) \supseteq \simp(\mathcal{T}')$.
\end{corollary}
\begin{proof}
The set $\simp(\mathcal{T})$ consists of precisely the $(d+1)$-tuples giving internal $d$-simplices of $\mathcal{T}$ by Lemma \ref{lem-int-odd}, while $\mathrm{sup}_{d}\mathcal{T}$ consists of all $d$-simplices of $\mathcal{T}$. It is then clear that $\mathrm{sup}_{d}\mathcal{T} \supseteq \mathrm{sup}_{d}\mathcal{T}'$ if and only if $\simp(\mathcal{T}) \supseteq \simp(\mathcal{T}')$, since boundary $d$-simplices are contained in every triangulation.
\end{proof}

\begin{remark}
Our interpretations of the higher Stasheff--Tamari orders in odd dimensions generalise the interpretation from \cite{fr} of the higher Stasheff--Tamari orders in dimension three using the corresponding persistent graph.
\end{remark}

As a consequence of Corollary \ref{thm-hst2-odd}, we obtain embeddings of the second higher Stasheff--Tamari posets into Boolean lattices of minimal rank, as in \cite{thomas}. Describing these embeddings in terms of submersion sets and supermersion sets makes them transparent.

\begin{corollary}
There is an embedding \[\iota\colon\mathcal{S}_{2}(m,2d+1) \hookrightarrow 2^{\binom{m-d-2}{d+1}},\] where the usual order on the Boolean lattice $2^{\binom{m-d-2}{d+1}}$ is reversed.
\end{corollary}
\begin{proof}
Define $\iota\colon \mathcal{S}_{2}(m,2d+1) \hookrightarrow 2^{\chainset{m}{d}}$ by \[\iota(\mathcal{T}):=\simp(\mathcal{T}).\] This is an embedding by Lemma \ref{lem-dey}. Then $\#\chainset{m}{d}=\binom{m-d-2}{d+1}$.
\end{proof}

One can use a similar technique to embed $\mathcal{S}_{2}(m,2d)$ into the smallest possible Boolean lattice. Such an embedding is realised if one restricts submersion sets to internal $d$-simplices which do not lie on the lower facets of $C(m,2d+1)$.

\begin{proposition}
There is an embedding \[\iota\colon\mathcal{S}_{2}(m,2d) \hookrightarrow 2^{\binom{m-d-1}{d+1}},\] where the order on the Boolean lattice $2^{\binom{m-d-1}{d+1}}$ is as usual.
\end{proposition}
\begin{proof}
If a $(d+1)$-tuple $A$ is either such that $|A|_{2d+1}$ is on the lower facets of $C(m,2d+1)$, or such that $|A|_{2d}$ lies in the facets of $C(m,2d)$, then $|A|_{2d}$ is submerged by every triangulation of $C(m,2d)$. Such $d$-simplices can therefore be ignored. The internal $d$-simplices of $C(m,2d)$ which do not lie on the lower facets of $C(m,2d+1)$ are precisely those with vertex tuples in $\chainset{m+1}{d}$. Hence define \[\iota(\mathcal{T}):=\{|A| \in \mathrm{sub}_{d}\mathcal{T} \mid A \in \chainset{m+1}{d}\}\] for a triangulation $\mathcal{T}$ of $C(m,2d)$. Then since $\#\chainset{m+1}{d}=\binom{m-d-1}{d+1}$, this gives our desired embedding. That this is an injection follows, of course, from the characterisation of the second higher Stasheff--Tamari order in terms of submersion sets from \cite[Proposition 2.15]{er}.
\end{proof}

The second order has the following interpretation on maximal chains of tilting modules.

\begin{theorem}\label{thm-alg-odd-hst2}
Given two triangulations $\mathcal{T}, \mathcal{T}' \in \mathcal{S}(n+2d,2d+1)$ corresponding to equivalence classes of maximal chains of tilting modules $[C], [C'] \in \widetilde{\mathcal{CT}}(A_{n}^{d})$, then $\mathcal{T} \leqslant_{2} \mathcal{T}'$ if and only if $\Sigma(C) \supseteq  \Sigma(C')$.
\end{theorem}
\begin{proof}
Let $\mathcal{T}$ be a triangulation of $C(n+2d,2d)$ corresponding to an equivalence class $[C]$ of maximal chains of tilting $A_{n}^{d}$-modules. By Theorem \ref{odd-dim-triangs}, we have that $M_{A} \in \Sigma(C)$ if and only if either $A \in \simp(\mathcal{T})$, or $M_{A}$ is projective or injective. Since the projectives and the injectives are tilting modules in every maximal chain, the result follows from Corollary \ref{thm-hst2-odd}.
\end{proof}

%The orders in odd dimensions are somehow too similar for separate examples
\begin{example}\label{ex-odd}
We continue considering triangulations of $C(6,3)$, as in Example \ref{ex-odd-desc}. We remind the reader that this cyclic polytope has six triangulations $\mathcal{T}_{l}, \mathcal{T}_{1}, \mathcal{T}_{2}, \mathcal{T}'_{1}, \mathcal{T}'_{2}, \mathcal{T}_{u}$, where
\[\begin{tabular}{lll}
$\simp(\mathcal{T}_{l})=\{24,25,35\}$; & $\simp(\mathcal{T}_{1})=\{24,25\}$; & $\simp(\mathcal{T}_{2})=\{24\}$; \\
\\
$\simp(\mathcal{T}'_{1})=\{25,35\}$; & $\simp(\mathcal{T}'_{2})=\{35\}$; & $\simp(\mathcal{T}_{u})= \emptyset$.
\end{tabular}\]

By Corollary \ref{thm-hst2-odd}, we have that $\mathcal{T}_{l} \leqslant_{2} \mathcal{T}_{1} \leqslant_{2} \mathcal{T}_{2} \leqslant_{2} \mathcal{T}_{u}$ and $\mathcal{T}_{l} \leqslant_{2} \mathcal{T}'_{1} \leqslant_{2} \mathcal{T}'_{2} \leqslant_{2} \mathcal{T}_{u}$. Moreover, it can be seen that the two orders are the same in this case by Theorem \ref{thm-hst1-odd}, since the covering relations are single-step reverse inclusions.

Consider the algebra $A_{4}^{1}$, which is the path algebra of \[1 \rightarrow 2 \rightarrow 3 \rightarrow 4.\] This has Auslander--Reiten quiver as shown in Figure \ref{fig-ar-quiv}. These indecomposable modules correspond to the 2-tuples in Figure \ref{fig-tuple}, which give 1-simplices in $C(6,3)$. We highlight in red the indecomposable modules which are neither projective nor injective. By Theorem \ref{odd-dim-triangs}, these correspond to internal $1$-simplices of $C(6,3)$, and so control the higher Stasheff--Tamari orders by Theorem \ref{thm-hst1-odd} and Corollary \ref{thm-hst2-odd}.

\begin{figure}
\caption{The Auslander--Reiten quiver of $A_{4}$.}\label{fig-ar-quiv}
\[\begin{tikzpicture}
\node(4) at (9,0) {${\tcs{4}}$};
\node(3) at (6,0) {\color{red} ${\tcs{3}}$};
\node(2) at (3,0) {\color{red} ${\tcs{2}}$};
\node(1) at (0,0) {${\tcs{1}}$};

\node(21) at (1.5,1.5) {${\tcs{2\\1}}$};
\node(32) at (4.5,1.5) {\color{red} ${\tcs{3\\2}}$};
\node(43) at (7.5,1.5) {${\tcs{4\\3}}$};

\node(432) at (6,3) {${\tcs{4\\3\\2}}$};
\node(321) at (3,3) {${\tcs{3\\2\\1}}$};

\node(4321) at (4.5,4.5) {${\tcs{4\\3\\2\\1}}$};

\draw[->] (1) -- (21);
\draw[->] (2) -- (32);
\draw[->] (3) -- (43);

\draw[->] (21) -- (321);
\draw[->] (32) -- (432);

\draw[->] (321) -- (4321);

\draw[->] (21) -- (2);
\draw[->] (32) -- (3);
\draw[->] (43) -- (4);

\draw[->] (321) -- (32);
\draw[->] (432) -- (43);

\draw[->] (4321) -- (432);
\end{tikzpicture}\]
\end{figure}
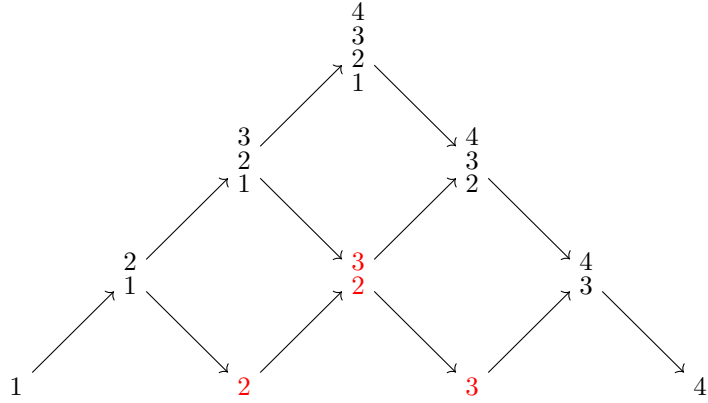

\begin{figure}
\caption{2-tuples corresponding to indecomposable $A_{4}$-modules.}\label{fig-tuple}
\[\begin{tikzpicture}
\node(4) at (0,0) {$M_{13}$};
\node(3) at (3,0) {\color{red} $M_{24}$};
\node(2) at (6,0) {\color{red} $M_{35}$};
\node(1) at (9,0) {$M_{46}$};

\node(34) at (1.5,1.5) {$M_{14}$};
\node(23) at (4.5,1.5) {\color{red} $M_{25}$};
\node(12) at (7.5,1.5) {$M_{36}$};

\node(234) at (3,3) {$M_{15}$};
\node(123) at (6,3) {$M_{26}$};

\node(1234) at (4.5,4.5) {$M_{16}$};

\draw[->] (4) -- (34);
\draw[->] (3) -- (23);
\draw[->] (2) -- (12);

\draw[->] (34) -- (234);
\draw[->] (23) -- (123);

\draw[->] (234) -- (1234);

\draw[->] (34) -- (3);
\draw[->] (23) -- (2);
\draw[->] (12) -- (1);

\draw[->] (234) -- (23);
\draw[->] (123) -- (12);

\draw[->] (1234) -- (123);
\end{tikzpicture}\]
\end{figure}
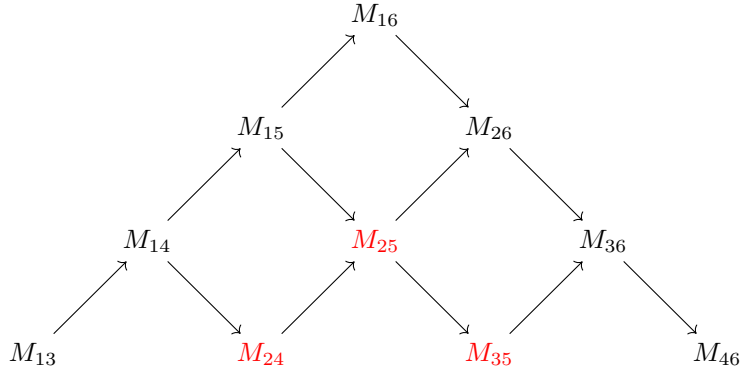

There are six maximal chains of tilting modules up to equivalence, corresponding to the respective triangulations. These are shown in Figure \ref{fig-max-chains}. Here $C_{l}$ corresponds to $\mathcal{T}_{l}$, and so on. We again highlight in red the indecomposable modules which are neither projective nor injective.

\begin{figure}
\caption{Maximal chains of tilting $A_{3}$-modules.}\label{fig-max-chains}
\begin{align*}
C_{l} :=\, &{\tcs{1}} \oplus {\tcs{2\\1}} \oplus {\tcs{3\\2\\1}} \oplus {\tcs{4\\3\\2\\1}} \leadsto {\color{red} {\tcs{2}}} \oplus {\tcs{2\\1}} \oplus {\tcs{3\\2\\1}} \oplus {\tcs{4\\3\\2\\1}} \leadsto {\color{red} {\tcs{2}}} \oplus {\color{red} {\tcs{3\\2}}} \oplus {\tcs{3\\2\\1}} \oplus {\tcs{4\\3\\2\\1}} \leadsto {\color{red} {\tcs{2}}} \oplus {\color{red} {\tcs{3\\2}}} \oplus {\tcs{4\\3\\2}} \oplus {\tcs{4\\3\\2\\1}} \leadsto \\
 &{\color{red} {\tcs{3}}} \oplus {\color{red} {\tcs{3\\2}}} \oplus {\tcs{4\\3\\2}} \oplus {\tcs{4\\3\\2\\1}} \leadsto {\color{red} {\tcs{3}}} \oplus {\tcs{4\\3}} \oplus {\tcs{4\\3\\2}} \oplus {\tcs{4\\3\\2\\1}} \leadsto {\tcs{4}} \oplus {\tcs{4\\3}} \oplus {\tcs{4\\3\\2}} \oplus {\tcs{4\\3\\2\\1}}\,, \\
C_{1} :=\, &{\tcs{1}} \oplus {\tcs{2\\1}} \oplus {\tcs{3\\2\\1}} \oplus {\tcs{4\\3\\2\\1}} \leadsto {\color{red} \tcs{2}} \oplus {\tcs{2\\1}} \oplus {\tcs{3\\2\\1}} \oplus {\tcs{4\\3\\2\\1}} \leadsto {\color{red} {\tcs{2}}} \oplus {\color{red} {\tcs{3\\2}}} \oplus {\tcs{3\\2\\1}} \oplus {\tcs{4\\3\\2\\1}} \leadsto {\color{red} {\tcs{2}}} \oplus {\color{red} {\tcs{3\\2}}} \oplus {\tcs{4\\3\\2}} \oplus {\tcs{4\\3\\2\\1}} \leadsto \\
 & {\color{red} {\tcs{2}}} \oplus {\tcs{4}} \oplus {\tcs{4\\3\\2}} \oplus {\tcs{4\\3\\2\\1}} \leadsto {\tcs{4}} \oplus {\tcs{4\\3}} \oplus {\tcs{4\\3\\2}} \oplus {\tcs{4\\3\\2\\1}}\,, \\
C'_{1} :=\, &{\tcs{1}} \oplus {\tcs{2\\1}} \oplus {\tcs{3\\2\\1}} \oplus {\tcs{4\\3\\2\\1}} \leadsto {\tcs{1}} \oplus {\color{red} {\tcs{3}}} \oplus {\tcs{3\\2\\1}} \oplus {\tcs{4\\3\\2\\1}} \leadsto {\color{red} {\tcs{3}}} \oplus {\color{red} {\tcs{3\\2}}} \oplus {\tcs{3\\2\\1}} \oplus {\tcs{4\\3\\2\\1}} \leadsto {\color{red} {\tcs{3}}} \oplus {\color{red} {\tcs{3\\2}}} \oplus {\tcs{4\\3\\2}} \oplus {\tcs{4\\3\\2\\1}} \leadsto \\
 & {\color{red} {\tcs{3}}} \oplus {\tcs{4\\3}} \oplus {\tcs{4\\3\\2}} \oplus {\tcs{4\\3\\2\\1}} \leadsto {\tcs{4}} \oplus {\tcs{4\\3}} \oplus {\tcs{4\\3\\2}} \oplus {\tcs{4\\3\\2\\1}}\,, \\
C_{2} :=\, &{\tcs{1}} \oplus {\tcs{2\\1}} \oplus {\tcs{3\\2\\1}} \oplus {\tcs{4\\3\\2\\1}} \leadsto {\color{red} {\tcs{2}}} \oplus {\tcs{2\\1}} \oplus {\tcs{3\\2\\1}} \oplus {\tcs{4\\3\\2\\1}} \leadsto {\color{red} {\tcs{2}}} \oplus {\tcs{4}} \oplus {\tcs{2\\1}} \oplus {\tcs{4\\3\\2\\1}} \leadsto {\color{red} {\tcs{2}}} \oplus {\tcs{4}} \oplus {\tcs{4\\3\\2}} \oplus {\tcs{4\\3\\2\\1}} \leadsto \\
 &{\tcs{4}} \oplus {\tcs{4\\3}} \oplus {\tcs{4\\3\\2}} \oplus {\tcs{4\\3\\2\\1}}\,, \\
C'_{2} :=\, &{\tcs{1}} \oplus {\tcs{2\\1}} \oplus {\tcs{3\\2\\1}} \oplus {\tcs{4\\3\\2\\1}} \leadsto {\tcs{1}} \oplus {\color{red} {\tcs{3}}} \oplus {\tcs{3\\2\\1}} \oplus {\tcs{4\\3\\2\\1}} \leadsto {\color{red} {\tcs{3}}} \oplus {\tcs{1}} \oplus {\tcs{4\\3}} \oplus {\tcs{4\\3\\2\\1}} \leadsto {\color{red} {\tcs{3}}} \oplus {\tcs{4\\3}} \oplus {\tcs{4\\3\\2}} \oplus {\tcs{4\\3\\2\\1}} \leadsto \\
 &{\tcs{4}} \oplus {\tcs{4\\3}} \oplus {\tcs{4\\3\\2}} \oplus {\tcs{4\\3\\2\\1}}\,, \\
C_{u} :=\, &{\tcs{1}} \oplus {\tcs{2\\1}} \oplus {\tcs{3\\2\\1}} \oplus {\tcs{4\\3\\2\\1}} \leadsto {\tcs{1}} \oplus {\tcs{4}} \oplus {\tcs{2\\1}} \oplus {\tcs{4\\3\\2\\1}} \leadsto {\tcs{1}} \oplus {\tcs{4}} \oplus {\tcs{4\\3}} \oplus {\tcs{4\\3\\2\\1}} \leadsto {\tcs{4}} \oplus {\tcs{4\\3}} \oplus {\tcs{4\\3\\2}} \oplus {\tcs{4\\3\\2\\1}}\,.
\end{align*}
\end{figure}

An example of an increasing elementary polygonal deformation is that between $C_{1}$ and $C_{2}$. This is shown in Figure \ref{fig-ex-flip}. This increasing elementary polygonal deformation reflects the fact that $\mathcal{T}_{1} \lessdot_{1} \mathcal{T}_{2}$, by Theorem \ref{thm-alg-odd-hst1}.

\begin{figure}
\caption{An increasing elementary polygonal deformation between the maximal chains of tilting modules $C_{1}$ and $C_{2}$.}\label{fig-ex-flip}
\[\begin{tikzcd}
&& {\tcs{4}} \oplus {\tcs{4\\3}} \oplus {\tcs{4\\3\\2}} \oplus {\tcs{4\\3\\2\\1}} && \\
&& {\color{red} {\tcs{2}}} \oplus {\tcs{4}} \oplus {\tcs{4\\3\\2}} \oplus {\tcs{4\\3\\2\\1}} \ar[u,rightsquigarrow] && \\
& {\color{red} {\tcs{2}}} \oplus {\color{red} {\tcs{3\\2}}} \oplus {\tcs{4\\3\\2}} \oplus {\tcs{4\\3\\2\\1}} \ar[ur,rightsquigarrow] && & \\
& {} \ar[rr,Rightarrow] &&  {\color{red} {\tcs{2}}} \oplus {\tcs{4}} \oplus {\tcs{2\\1}} \oplus {\tcs{4\\3\\2\\1}} \ar[uul,rightsquigarrow] & \\
& {\color{red} {\tcs{2}}} \oplus {\color{red} {\tcs{3\\2}}} \oplus {\tcs{3\\2\\1}} \oplus {\tcs{4\\3\\2\\1}} \ar[uu,rightsquigarrow] &&  & \\
&& {\color{red} {\tcs{2}}} \oplus {\tcs{2\\1}} \oplus {\tcs{3\\2\\1}} \oplus {\tcs{4\\3\\2\\1}} \ar[ul,rightsquigarrow] \ar[uur,rightsquigarrow] &&\\
&& {\tcs{1}} \oplus {\tcs{2\\1}} \oplus {\tcs{3\\2\\1}} \oplus {\tcs{4\\3\\2\\1}}\,. \ar[u,rightsquigarrow] &&
\end{tikzcd}\]
\end{figure}

\begin{figure}
\caption{The poset $\mathcal{S}_{1}(6,3)=\mathcal{S}_{2}(6,3)$.}\label{fig-odd-poset}
\[
\begin{tikzpicture}

% Bottom triangulation

\node(24l) at (0,0) {\tiny $M_{24}$};
\node(25l) at (0.7,0.7) {\tiny $M_{25}$};
\node(35l) at (1.4,0) {\tiny $M_{35}$};

\draw[->] (24l) -- (25l);
\draw[->] (25l) -- (35l);
\draw[red,thick] plot [smooth cycle] coordinates {(0.2,-0.2) (-0.3,-0.2) (-0.3,0.2) (0.7,1) (1.7,0.2) (1.7,-0.2) (1.2,-0.2) (0.7,0.4)};

% Bottom left triangulation

\node(24bl) at (-3,3) {\tiny $M_{24}$};
\node(25bl) at (-2.3,3.7) {\tiny $M_{25}$};
\node(35bl) at (-1.6,3) {\tiny $M_{35}$};

\draw[->] (24bl) -- (25bl);
\draw[->] (25bl) -- (35bl);
\draw[red,thick] plot [smooth cycle] coordinates {(-2.8,2.8) (-3.3,2.8) (-3.3,3.2) (-2.4,4) (-1.9,3.7)};

% Bottom right triangulation

\node(24br) at (3,3) {\tiny $M_{24}$};
\node(25br) at (3.7,3.7) {\tiny $M_{25}$};
\node(35br) at (4.4,3) {\tiny $M_{35}$};

\draw[->] (24br) -- (25br);
\draw[->] (25br) -- (35br);
\draw[red,thick] plot [smooth cycle] coordinates {(4.2,2.8) (4.7,2.8) (4.7,3.2) (3.8,4) (3.3,3.7)};

% Top left triangulation

\node(24tl) at (-3,6) {\tiny $M_{24}$};
\node(25tl) at (-2.3,6.7) {\tiny $M_{25}$};
\node(35tl) at (-1.6,6) {\tiny $M_{35}$};

\draw[->] (24tl) -- (25tl);
\draw[->] (25tl) -- (35tl);
\draw[red,thick] plot [smooth cycle] coordinates {(-2.7,5.8) (-3.3,5.8)  (-3.3,6.2) (-2.7,6.2)};

% Top right triangulation

\node(24tr) at (3,6) {\tiny $M_{24}$};
\node(25tr) at (3.7,6.7) {\tiny $M_{25}$};
\node(35tr) at (4.4,6) {\tiny $M_{35}$};

\draw[->] (24tr) -- (25tr);
\draw[->] (25tr) -- (35tr);
\draw[red,thick] plot [smooth cycle] coordinates {(4.1,5.8) (4.7,5.8)  (4.7,6.2) (4.1,6.2)};

% Top triangulation

\node(24u) at (0,9) {\tiny $M_{24}$};
\node(25u) at (0.7,9.7) {\tiny $M_{25}$};
\node(35u) at (1.4,9) {\tiny $M_{35}$};

\draw[->] (24u) -- (25u);
\draw[->] (25u) -- (35u);

% Arrows between triangulations

\draw[->] (-0.2,0.7) -- (-2.2,2.7);
\draw[->] (1.6,0.7) -- (3.6,2.7);
\draw[->] (-2.3,4.2) -- (-2.3,5.7);
\draw[->] (3.7,4.2) -- (3.7,5.7);
\draw[->] (-2,7) -- (-0.3,8.7);
\draw[->] (3.4,7) -- (1.7,8.7);

% Triangulation labels

\node at (-0.8,0.3) {$\mathcal{T}_{l}$};
\node at (-3.8,3.3) {$\mathcal{T}_{1}$};
\node at (2.2,3.3) {$\mathcal{T}'_{1}$};
\node at (-3.8,6.3) {$\mathcal{T}_{2}$};
\node at (2.2,6.3) {$\mathcal{T}'_{2}$};
\node at (-0.8,9.3) {$\mathcal{T}_{u}$};

\end{tikzpicture}
\]
\end{figure}
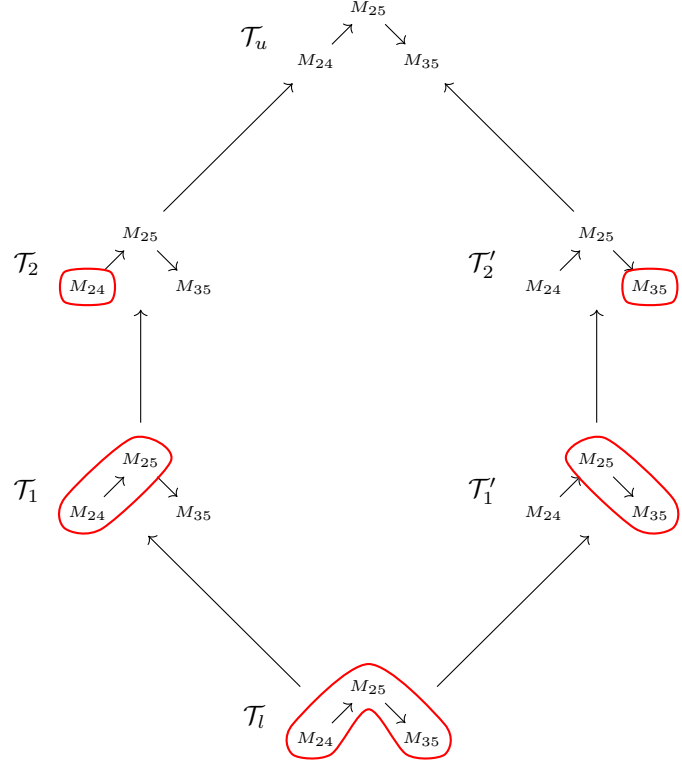

One can read off the second order from the maximal chains in Figure \ref{fig-max-chains}. All maximal chains contain as indecomposable summands the modules \[{\tcs{1}}\, ,\,{\tcs{2\\1}}\, ,\, {\tcs{3\\2\\1}}\, ,\, {\tcs{4\\3\\2\\1}}\, ,\, {\tcs{4\\3\\2}}\, ,\, {\tcs{4\\3}}\, ,\, {\tcs{4}},\] since these are the projectives and injectives. Hence, these summands can be ignored. Then the second order is determined by reverse inclusion with respect to the indecomposable summands \[{\tcs{2}}\, ,\, {\tcs{3\\2}}\, ,\, {\tcs{3}}.\] These correspond to the internal 1-simplices $|24|,|25|,|35|$ respectively.

In Figure \ref{fig-odd-poset} we illustrate the poset $\mathcal{S}_{1}(6,3)=\mathcal{S}_{2}(6,3)$. The circled modules correspond to the internal 1-simplices of the triangulation, and hence to the summands of the corresponding maximal chain of tilting modules which are neither projective nor injective.
\end{example}

\subsection{Maximal green sequences}\label{sect-max-green}

In this section, we use the cluster-tilting framework to explain the connection between triangulations of odd-dimensional cyclic polytopes and the maximal green sequences of Keller \cite{kel-green,dem-kel}.

Let $\Lambda$ be a $d$-representation-finite $d$-hereditary $K$-algebra, where $K$ is a field. We define a \emph{$d$-maximal green sequence} for $\Lambda$ to be a sequence $(T_{1}, \dots, T_{r})$ of cluster-tilting objects in $\mathcal{V}_{\Lambda}$ such that  $T_{1}=\Lambda$, $T_{r}=\Lambda[d]$, and, for $i \in [r-1]$, $T_{i+1}$ is a left mutation of $T_{i}$. Let $\mathcal{MG}_{d}(\Lambda)$ denote the set of $d$-maximal green sequences of $\Lambda$. Given a $d$-maximal green sequence $G$, we denote the set of indecomposable summands of cluster-tilting objects occurring in $G$ by $\Sigma(G)$. We write $G \sim G'$ if and only if $\Sigma(G)=\Sigma(G')$. As before, we use $\widetilde{\mathcal{MG}}_{d}(\Lambda)$ to denote the set of equivalence classes of $\mathcal{MG}_{d}(\Lambda)$ under the relation $\sim$.

\begin{remark}
The discrepancies mentioned in Remark \ref{rmk-clus-tilt} do not affect the notion of $d$-maximal green sequences. It follows from \cite[Theorem 5.7(2)]{ot} that any mutation of an Oppermann--Thomas cluster-tilting object is an Oppermann--Thomas cluster-tilting object. Since $\Lambda$ is an Oppermann--Thomas cluster-tilting object, any cluster-tilting object occurring in a $d$-maximal green sequence corresponds to an Oppermann--Thomas cluster-tilting object.
\end{remark}

Since cluster-tilting objects for $A_{n}^{d}$ are in bijection with tilting modules for $A_{n+1}^{d}$ by \cite[Theorem 1.1 and Theorem 1.2]{ot}, we have that the set $\widetilde{\mathcal{MG}}_{d}(A_{n}^{d})$ is in bijection with $\widetilde{\mathcal{CT}}(A_{n+1}^{d})$. Hence we obtain the following theorem, which corresponds to Theorem \ref{odd-dim-triangs} in the tilting framework.

\begin{theorem}\label{thm-max-green-desc}
There is a bijection between triangulations of $C(n+2d+1,2d+1)$ and $\widetilde{\mathcal{MG}}_{d}(A_{n}^{d})$. Moreover, if a triangulation $\mathcal{T}$ of $C(n+2d+1,2d+1)$ corresponds to a $d$-maximal green sequence $[G] \in \widetilde{\mathcal{MG}}_{d}(A_{n}^{d})$, then
\begin{enumerate}
\item there is a bijection between mutations in $G$ and $(2d+1)$-simplices of $\mathcal{T}$; and
\item  there is a bijection between the internal $d$-simplices of $\mathcal{T}$ and elements of $\Sigma(G)$ which are neither projectives nor shifted projectives. 
\end{enumerate}
\end{theorem}

\begin{remark}
Theorem \ref{thm-class-odd-dim} therefore also classifies $d$-maximal green sequences for $A_{n}^{d}$ up to equivalence.
\end{remark}

Moreover, one can interpret the higher Stasheff--Tamari orders on the equivalence classes of $d$-maximal green sequences analogously to Theorem \ref{thm-alg-odd-hst1} and Theorem \ref{thm-alg-odd-hst2}.

\begin{theorem}\label{thm-max-green}
Let $\mathcal{T}, \mathcal{T}' \in \mathcal{S}(n+2d+1,2d+1)$ correspond to equivalence classes of $d$-maximal green sequences $[G],[G'] \in \widetilde{\mathcal{MG}}_{d}(A_{n}^{d})$.
\begin{enumerate}
\item We have that $\mathcal{T}\lessdot_{1}\mathcal{T}'$ if and only if there are equivalence class representatives $\widehat{G} \in [G]$ and $\widehat{G}' \in [G']$ such that $\widehat{G}'$ is an increasing elementary polygonal deformation of $\widehat{G}$.
\item We have that $\mathcal{T} \leqslant_{2} \mathcal{T}'$ if and only if $\Sigma(G) \supseteq  \Sigma(G')$.
\end{enumerate} 
\end{theorem}

Since it is known for dimension 3 that the higher Stasheff--Tamari orders are equal and are lattices \cite[Theorem 4.9 and Theorem 4.10]{er}, we have the following corollary.

\begin{corollary}\label{cor-a-green-lat}
Given ($1$-)maximal green sequences $[G], [G'] \in \widetilde{\mathcal{MG}}_{1}(A_{n})$ we have that $\Sigma(G) \supseteq \Sigma(G')$ if and only if there exists a set of equivalence classes of maximal green sequences $[G_{0}], \dots, [G_{r}]$, where $G_{0} \sim G$, $G_{r} \sim G'$, and, for each $i \in [r]$, there exist $\widehat{G}_{i} \in [G_{i}]$ and $\widehat{G}_{i - 1} \in [G_{i - 1}]$ such that $\widehat{G}_{i}$ is an increasing elementary polygonal deformation of $\widehat{G}_{i - 1}$. 

Moreover, the set $\widetilde{\mathcal{MG}}_{1}(A_{n})$ forms a lattice under either of these equivalent orders.
\end{corollary}

\begin{remark}
In independent work, \cite{gorsky_phd,gorsky_note,gorsky_forthcoming} Gorsky defines two orders on the set of equivalence classes of maximal green sequences of a Dynkin quiver, using combinatorics of the associated Coxeter group, and proves that they are the same. For type $A$ quivers, these coincide with the two higher Stasheff--Tamari orders considered here.
\end{remark}

\begin{remark}\label{ex-counter-green-lat}
It is not in general true that the set of equivalence classes of maximal green sequences of a finite-dimensional algebra is a lattice. For example, the preprojective algebra of $A_{2}$ only has two maximal green sequences. These are not equivalent to each other, and nor are they related by either of the relations described above. Hence in this case the set of maximal green sequences modulo equivalence is not a lattice.

One might wonder whether the set of equivalence classes of maximal green sequences is a lattice for other hereditary algebras. However, computer calculations reveal that the set of equivalence classes of maximal green sequences of the path algebra of Dynkin type $D_{4}$ is not a lattice.
\end{remark}

\begin{remark}
A common way of considering a maximal green sequence for $d=1$ is as a chain of torsion classes \cite{nagao}. A natural question to ask, therefore, is whether there exists an analogous description for $d>1$.

For a cluster-tilting object $T$ for $A_{n}^{d}$, the associated $d$-torsion class ought to be $T^{\bot} \cap \mathrm{add}\, M^{(d,n)}$. Indeed, this class corresponds to the internal $d$-simplices of the supermersion set of the associated triangulation of $C(n+2d+1,2d)$, excluding internal $d$-simplices belonging to the upper triangulation, which are in every supermersion set.

But there are two problems here. Firstly, maximal chains of these $d$-torsion classes are maximal chains in the second order, whereas $\mathcal{MG}_{d}(A_{n}^{d})$ consists of maximal chains in the first order. The two will only be the same if the two higher Stasheff--Tamari orders are equal. Secondly, the $d$-torsion classes that are generated in this way do not satisfy any definitions of higher torsion classes that have appeared so far in the literature, such as \cite{jorg-torsion,mcm-supp-tilt}.
\end{remark}

\section{Computations}\label{sect-comp}

We used the description of the higher Stasheff--Tamari orders in Theorems \ref{thm-int-even-comb} and \ref{thm-int-odd-comb} to construct the posets in Sage. We then tested the Edelman--Reiner conjecture. The conjecture held in all cases tested, as detailed in Figure \ref{fig-table-equiv}. The number of triangulations of a cyclic polytope grows rapidly \cite{jos-kast}, so only relatively small examples are computable.

\begin{figure}
\caption{Equivalence of the higher Stasheff--Tamari orders for $C(c+\delta,\delta)$.}\label{fig-table-equiv}
\[\begin{tabular}{c|rrrrrrrrrrrrrrr}
$c\backslash\delta$ & 4 & 5 & 6 & 7 & 8 & 9 & 10 & 11 & 12 & 13 & 14 & 15 & 16 \\
\hline
4 & \cmark & \cmark & \cmark & \cmark & \cmark & \cmark & \cmark & \cmark & \cmark & \cmark & \cmark & \cmark & \cmark \\
5 & \cmark & \cmark & \cmark &&&&&&&&&& \\
6 & \cmark &&&&&&&&&&&&
\end{tabular}\]
\end{figure}

The authors of \cite{err} use computer calculations to show that $\mathcal{S}_{2}(9,4)$ is not a lattice. By our calculations, $\mathcal{S}_{2}(9,4)$ and $\mathcal{S}_{1}(9,4)$ are the same poset, which implies that $\mathcal{S}_{1}(9,4)$ is not a lattice either. A counter-example to the first order's being a lattice was not previously known \cite[\S 5]{rr}. We tested the lattice property of the first order in several other cases, as shown in Figure \ref{fig-table-lat}.

\begin{figure}
\caption{Lattice property for $\mathcal{S}_{1}(c+\delta,\delta)$.}\label{fig-table-lat}
\[\begin{tabular}{c|rrrrrrrrrrrrrrr}
$c\backslash\delta$ & 4 & 5 & 6 & 7 & 8 & 9 & 10 & 11 & 12 & 13 & 14 & 15 & 16 \\
\hline
4 & \cmark & \cmark & \cmark & \cmark & \cmark & \cmark & \cmark & \cmark & \cmark & \cmark & \cmark & \cmark & \cmark  \\
5 & \xmark & \xmark & \xmark & \xmark & \xmark &&&&&&&& \\
6 & \xmark &&&&&&&&&&&&
\end{tabular}\]
\end{figure}

\begin{figure}
\caption{$\mathcal{S}_{1}(14,10)=\mathcal{S}_{2}(14,10)$, generated using Sage.}
\[\includegraphics[scale=0.05]{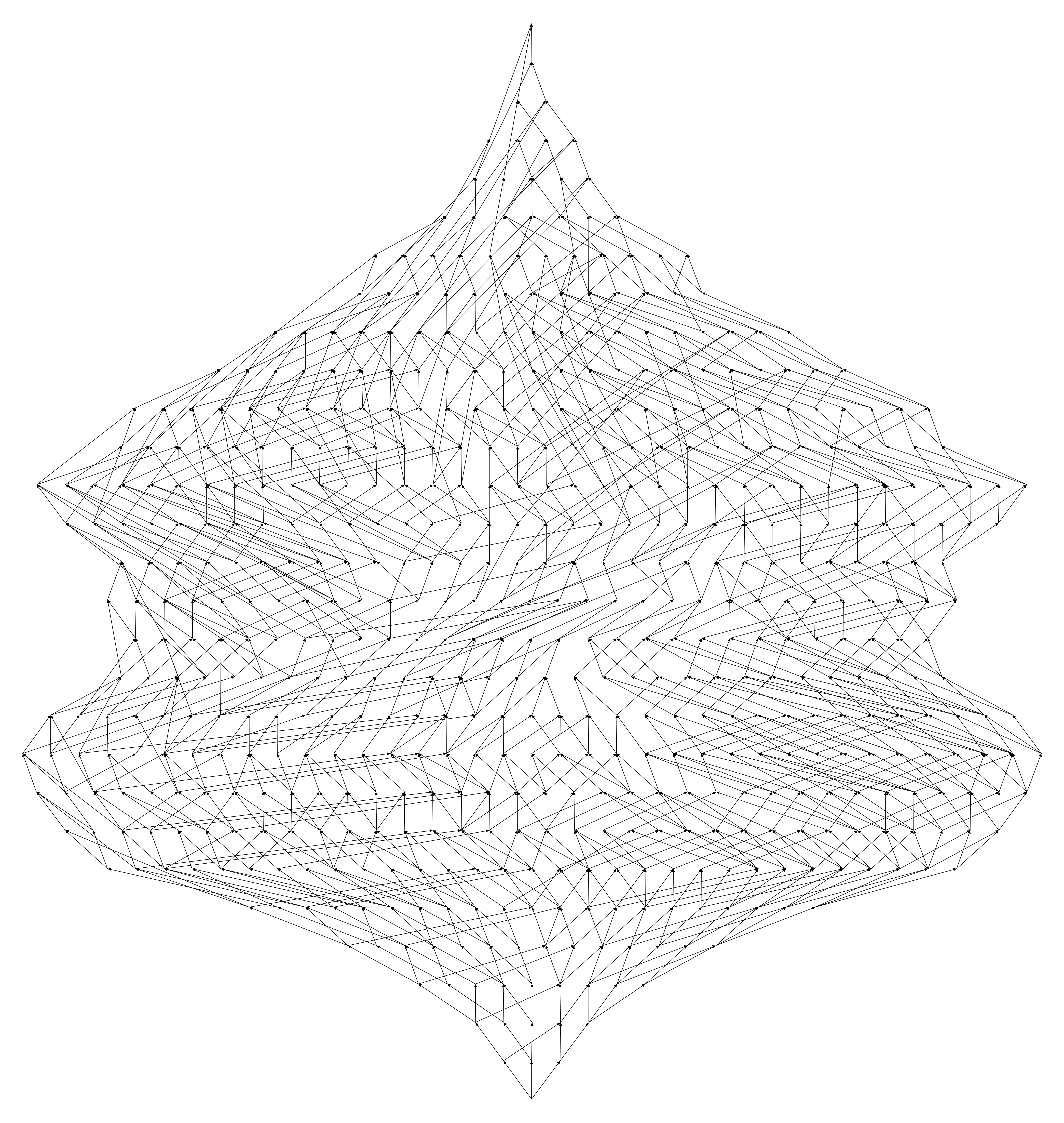}\]
\end{figure}

\printbibliography

\end{document}